\documentclass[]{amsart}
\usepackage{latexsym,xy,epstopdf}
\xyoption{all}
\usepackage{graphicx,amscd,amssymb,latexsym,subfigure,verbatim,hyperref,epsfig,supertabular}
\usepackage{color}

\usepackage{mathtools}

\DeclarePairedDelimiter\floor{\lfloor}{\rfloor}
\definecolor{gray}{rgb}{.5,.4,.5}

\newtheorem{defn0}{Definition}[section]
\newtheorem{prop0}[defn0]{Proposition}
\newtheorem{conj0}[defn0]{Conjecture}
\newtheorem{thm0}[defn0]{Theorem}
\newtheorem{lem0}[defn0]{Lemma}
\newtheorem{exercise0}[defn0]{Exercise}
\newtheorem{lemma}[defn0]{Lemma}
\newtheorem{corollary0}[defn0]{Corollary}
\newtheorem{example0}[defn0]{Example}
\newtheorem{remark0}[defn0]{Remark}
\newtheorem{que0}[defn0]{Question}

\newenvironment{defn}{\begin{defn0}}{\end{defn0}}
\newenvironment{prop}{\begin{prop0}}{\end{prop0}}
\newenvironment{conj}{\begin{conj0}}{\end{conj0}}
\newenvironment{thm}{\begin{thm0}}{\end{thm0}}

\newenvironment{lem}{\begin{lem0}}{\end{lem0}}

\newenvironment{cor}{\begin{corollary0}}{\end{corollary0}}
\newenvironment{exm}{\begin{example0}\rm}{\end{example0}}
\newenvironment{remark}{\begin{remark0}\rm}{\end{remark0}}
\newcommand{\V}{{\bf V }}
\renewcommand{\P}{{\mathbb{P}}}
\newcommand{\Z}{{\mathbb{Z}}}
\newcommand{\m}{{\mathcal M}}

\newcommand{\K}{{\mathbb{K}}}

\newcommand{\ZZ}{{\mathbb{Z}}}
\newcommand{\Oc}{{\mathcal{O}}}

\newcommand{\OPP}{{\mathcal{O}_{\mathbb{P}^1 \times \mathbb{P}^1}}}
\newcommand{\hh}{{H^0(\Oc_{\PP}(1,n))}}
\newcommand{\PP}{{\mathbb{P}^1 \times \mathbb{P}^1}}

\newcommand{\OPone}{{\mathcal{O}_{\mathbb{P}^1}}}

\newcommand{\im}{{\mathrm{im}}}
\newcommand{\codim}{{\mathrm{codim}}}
\newcommand{\depth}{{\mathrm{depth}}}
\newcommand{\coker}{{\mathrm{coker}}}

\newcommand{\mm}{{\mathfrak{m}}}

\newcommand\paren[1]{\left(#1\right)}

\newcommand\To{\Rightarrow}

\newcommand\lto{\longrightarrow}

\def\NN{\mathbb{N}}
\def\ZZ{\mathbb{Z}}
\def\Cc{\mathcal{C}}  

\def\Kc{\mathcal{K}}  

\def\pp{\mathfrak{p}}
\def\mm{\mathfrak{m}}
\def\aaa{\mathfrak{a}}
\def\bb{\mathfrak{b}}
\def\f{\textbf{f}}
\def\d{\textbf{d}}
\def\a{\textbf{a}}

\def\k.{\mathcal{K}_{\bullet}}

\DeclareMathOperator{\rad}{rad}

\DeclareMathOperator{\Syz}{Syz}
\DeclareMathOperator{\Span}{Span}
\DeclareMathOperator{\supp}{Supp}
\DeclareMathOperator{\corank}{corank}

\DeclareMathOperator{\dom}{dom}
\DeclareMathOperator{\cod}{cod}

\numberwithin{equation}{section}

\begin{document}
\title[The simplest minimal free resolutions in $\PP$]%
{The simplest minimal free resolutions in $\PP$}

\author{Nicol\'as Botbol}
\address{Departamento de Matem\'atica\\
FCEN, Universidad de Buenos Aires, Ciudad Universitaria, Pab.\ I, 
C1428EGA Buenos Aires, Argentina}
\email{nbotbol@dm.uba.ar}

\author{Alicia Dickenstein}
\thanks{Dickenstein is   supported by ANPCyT PICT 2016-0398,
 UBACYT 20020170100048BA, and CONICET PIP 11220150100473, Argentina}
\address{Departamento de Matem\'atica\\
FCEN, Universidad de Buenos Aires,  and IMAS (UBA-CONICET), Ciudad Universitaria, Pab.\ I, 
C1428EGA Buenos Aires, Argentina}
\email{alidick@dm.uba.ar}

\author{Hal Schenck}
\thanks{Schenck is supported by NSF 1818646, Fulbright FSP 5704}
\address{Department of Mathematics, Auburn University, Auburn AL 36849}
\email{hks0015@auburn.edu}


\begin{abstract}
We study the minimal bigraded free resolution of 
an ideal with three generators of the same bidegree, contained in the 
bihomogeneous maximal ideal $ \langle s,t\rangle \cap \langle u,v \rangle$ of the bigraded ring $\K[s,t;u,v]$. Our analysis
involves tools from algebraic geometry (Segre-Veronese varieties), 
classical commutative algebra (Buchsbaum-Eisenbud criteria for exactness, Hilbert-Burch theorem), and
homological algebra (Koszul homology, spectral sequences). We treat in detail the case in which the bidegree is $(1,n)$.
We connect our work to a conjecture of Fr\"oberg-Lundqvist 
on bigraded Hilbert functions, and close with a number of open problems.
\end{abstract}

\maketitle

\section{Introduction}
In this chapter, we consider bigraded minimal free resolutions in the first nontrivial case. Let 
$R=\K[s,t;u,v]$ be the bigraded polynomial ring, where $\{s,t\}$ are of degree
$(1,0)$ and $\{u,v\}$ are of degree $(0,1)$; $R$ is graded by
$\Z^2$. 
For ${\bf d}=(d_1,d_2)$, we consider a three dimensional
subspace $W = \Span \{f_0,f_1,f_2\}   \subseteq R_{\bf d}$, with
the additional constraint that 
\begin{equation}\label{eq:B}
I_W= \langle f_0,f_1,f_2\rangle \mbox{ satisfies }\sqrt{I_W} =  \langle s,t\rangle \cap \langle u,v \rangle.
\end{equation}
This generic condition arises from the  natural geometric condition of being {\em basepoint free}, defined in \S \ref{AGBackground} below.
We study the minimal free resolution of $I_W$ and we give precise results when ${\bf d} =(1,n)$ and $n \ge 3$.\newline

\begin{exm}~\label{11case}
For ${\bf d} = (1,1)$ 
the bigraded Betti numbers of $I_W$ are always
\begin{table}[ht]
\begin{supertabular}{|c|}
\hline 
  $0 \leftarrow I_W \leftarrow (-1,-1)^3 \stackrel{\partial_1}{\longleftarrow} \begin{array}{c}
( -1,-3)\\
 \oplus \\
 (-2,-2)^3\\
 \oplus \\
(-3,-1)\\
\end{array} \stackrel{\partial_2}{\longleftarrow} 
\begin{array}{c}
 (-2,-3)^2\\
 \oplus \\
(-3,-2)^2\\
 \end{array} \stackrel{\partial_3}{\longleftarrow} 
(-3,-3)\leftarrow 0$ \\
\hline 
\end{supertabular}
\vskip .1in
\end{table} \newline
\noindent The degree $(2,2)$ syzygies are Koszul. The first syzygies of
degree $(1,3)$ and $(3,1)$ involve only one set of variables, and arise from the vanishing of a determinant (see Lemma~\ref{Alicia}). The reader is encouraged to work out the remaining differentials. 
\end{exm}
\noindent The ${\bf d} = (1,2)$ case is more complex, and \cite{cds} shows that there are two 
possible bigraded minimal free resolutions for $I_W$. The resolution type is determined by how $\P(W) \subseteq \P(R_{1,2}) = \P^5$ 
meets the image $\Sigma_{1,2}$ of the Segre map $\P^1 \times  \P^2 \stackrel{\sigma_{1,2}}{\longrightarrow} \P^5$ of factorizable polynomials.

\subsection{Motivation from geometric modeling}\label{geomModel}
In geometric modeling, it is often useful to approximate a surface in $\P^3$ with a rational
surface of low degree. Most commonly the rational surfaces
used are $\PP$ or $\P^2$, and the resulting objects are known 
as {\it tensor product surfaces} and 
 {\it triangular surfaces}. See, for 
example, \cite{bdd},\cite{bc},\cite{c},\cite{cgz},\cite{sc}. 
An efficient way to compute the implicit equation is via approximation complexes 
\cite{hsv1},\cite{hsv2} which use syzygy data as input.

A tensor product surface is mapped to $\P^3$ by a four dimensional subspace $V \subseteq R_{\bf d}$. For ${\bf d}=(1,2)$, Zube describes the singular locus in  \cite{z1}, \cite{z2}, and in \cite{egl}, \cite{gl}, Elkadi-Galligo-L\^{e} 
use the geometry of a dual scroll to analyze the image. When
$\sqrt{I_V} = \langle s,t\rangle \cap \langle u,v \rangle$, \cite{ssv} shows that there are exactly six types of free resolution possible,
and analyzes the approximation complexes for the distinct
resolutions. Degan \cite{d} examines the situation when the subspace has basepoints. 
For a three dimensional subspace $W \subseteq V$, the ideals $I_V$ and $I_W$ are related via linkage in \cite{ssv}.

\subsection{Mathematical background}\label{MathBackground}
We start with a quick review of the cast of principal mathematical players, referring to \cite{ebig} and \cite{GeomSyz} for additional details. 
\subsubsection{Bigraded Betti numbers}\begin{defn}\label{BettiTable}
For a bihomogeneous ideal $I \subseteq R= \K[s,t;u,v]$ and ${\bf d} \in \Z^2$, the {\em bigraded Betti numbers} are
 \[
\beta_{i,{\bf a}} = \dim_{\K}Tor_i(R/I,\K)_{\bf a}.
\]
\end{defn}

For all nonnegative integers $i$ and all  bidegrees ${\bf a}$, $\beta_{i,{\bf a}}$ is the number of copies of $R(-{\bf a})$ appearing in the $i$th module of the minimal free resolution of $I$.
Since ${\bf a} \in \Z^2$, the $\beta_{i,{\bf a}}$ cannot be displayed in the $\Z$-graded Betti table format \cite{GeomSyz}. 
Bigraded regularity is studied in \cite{acd},\cite{hw},\cite{r}, and multigraded regularity was introduced in \cite{MS}. \newline

\begin{exm}~\label{smooth12}[\cite{cds}, Theorem~7.8]
Let $\Sigma_{1,2}$ be the Segre variety of $\P^1 \times \P^2 \subseteq \P^5= \P(U)$, where $U$ has basis $\{su^2, suv, sv^2,tu^2,tuv,tv^2\}$. Then $W \cap \Sigma_{1,2}$ is a smooth conic iff $I_W$ has the bigraded Betti numbers as below. 
\begin{table}[ht]
\begin{supertabular}{|c|}
\hline 
  $0 \leftarrow I_W \leftarrow (-1,-2)^3 \stackrel{\partial_1}{\longleftarrow} \begin{array}{c}
( -1,-6)\\
 \oplus \\
 (-2,-4)^3\\
 \oplus \\
(-3,-2)\\
\end{array} \stackrel{\partial_2}{\longleftarrow} 
\begin{array}{c}
 (-2,-6)^2\\
 \oplus \\
(-3,-4)^2\\
 \end{array} \stackrel{\partial_3}{\longleftarrow} 
(-3,-6)\leftarrow 0$ \\
\hline 
\end{supertabular}
\end{table}
\vskip -.1in
For example, $\beta_{1,(2,4)} = 3$ and $\beta_{2,(3,4)}=2$.
\end{exm}
\pagebreak

\subsubsection{Bigraded algebra and line bundles on $\PP$}\label{AGBackground}
As noted earlier, the constraint that $\sqrt{I_W}$ is the bihomogeneous maximal ideal in~\eqref{eq:B}
arises as a natural geometric condition, and we give a 
quick synopsis; for additional details, see \S V.I of \cite{h}. 

A line bundle ${\mathcal L}$ on the abstract variety $\PP$ is characterized by a choice of ${\bf d} \in \Z^2$ 
and we write $\OPP({\bf d})$ for ${\mathcal L}$. Although the global sections 
\[
H^0(\OPP({\bf d}))= R_{\bf d}
\]
are not functions on $\PP$, their ratios give well defined functions on $\PP$ and so zero sets of sections are
defined.

The upshot is that to realize $\PP$ as a subvariety of $\P^n$, we
choose an $n+1$ dimensional subspace $W \subseteq R_{\bf d}$ with
${\bf d} \in \Z^2_{>(0,0)}$. As long as the $f_i \in W$ do not
simultaneously vanish on $\PP$, this gives a regular map from $\PP$ to $\P^n$. 
The condition that the $f_i$ do not simultaneously vanish at a point
of $\PP$ is exactly the condition~\ref{eq:B}; in this situation $W$ is
said to be {\em basepoint free}. For example, if $W=\Span
\{su,sv,tu\}$, then $[(0:1),(0:1)] \in \PP$ is a basepoint of $W$.

\subsubsection{Koszul homology and bigraded Hilbert series}

 Koszul homology is defined and discussed in \S 2. We will give in Definition~\ref{def:hda} the precise conditions under which a basepoint free system of polynomials ${\bf f} = \{f_0,f_1,f_2\} \subseteq R_{\bf d}$ is {\em generic}. 
 In this case, we make a conjecture in \S 2 on the
bigraded Betti numbers, and show that it is equivalent to a recent
conjecture made by Fr\"oberg-Lundqvist in \cite{FL} on the bigraded Hilbert series of $R/I_W$.

\subsection{Roadmap of this chapter}
Below is an overview of the sections which make up this chapter. 

\begin{itemize}
\item In \S 2, we study the Koszul homology of $I$; the first homology encodes the non-Koszul first syzygies. 
The spectral sequence of the \v{C}ech-Koszul double complex has a
single $d_3$ differential, and we explain the connection to local cohomology $H^{\bullet}_B$.  We make a conjecture about the the first Koszul homology module for {\em generic} $W,$ and connect it to a conjecture of Fr\"oberg-Lundqvist \cite{FL} on the Hilbert series of generic bigraded ideals.
\vskip .1in
\item In \S 3, we use tools from commutative algebra such as the Hilbert-Burch theorem to shed additional light on the first syzygies. 
\vskip .1in
\item In \S 4, we connect the minimal free resolution to the image of the Segre variety $\Sigma_{1,n}$, obtaining canonical syzygies in certain degrees, without any assumptions on genericity. The geometry of $W \cap \Sigma_{1,n}$ plays a key role.
\vskip .1in
\item In \S 5 we prove results on higher Segre varieties, in
  particular about how the geometry of the intersection of $W$ with
  such varieties influences the free resolution. We close with a number of questions.
\end{itemize}

\section{Koszul homology $H_1(\k.(\f,R))$ and the generic case}\label{sec:d}

We start this section with an overview of Koszul homology.
We then prove in Theorem~\ref{thm:H1} a characterization of
the first Koszul homology associated to our ideal $I_W$ under the assumption~\eqref{eq:B}.
Corollary~\ref{cor:H1} then gives a concrete representation of $H_1$, that we make explicit
in Examples~\ref{ex:maps} and~\ref{ex:maps6} for factorizable polynomials. Section~\ref{ssec:gencase}
treats the generic case (see Definition~\ref{def:hda}). In this case, we specify some values of 
the dimensions $(H_1)_{\bf a}$ for $\a \in \Z^2_{\ge 0}$ and we state Conjecture~\ref{conj:ANH}
about these dimensions.
We prove in Proposition~\ref{prop:conjs} that our Conjecture is equivalent to  Conjecture~\ref{conj:RS}  by Ralf Fr\"oberg and Samuel Lundqvist (Conjecture~8 in~\cite{FL}).  

\subsection{Koszul Homology}  \label{ssec:K} 
For notation, we write $\aaa=\langle s,t \rangle$, $\bb= \langle u,v
\rangle$, $B=\aaa \cap \bb$ and $\mm=R_+=(s,t,u,v)$
\begin{defn}
For a sequence of polynomials $\f=\{f_0,\ldots,f_m\}$ the Koszul
complex $\k.:=\k.(\f,R)$ is the complex
\[
  \cdots \longrightarrow \Lambda^j(R^{m+1})
  \stackrel{\delta_j}{\longrightarrow} \Lambda^{j-1}(R^{m+1}) \longrightarrow
  \cdots
\]
where
\[
  \delta_j(e_{n_1}\wedge \cdots \wedge e_{n_j}) \mapsto \sum\limits_{i=1}^j
  (-1)^i f_{n_i}   \cdot (e_{n_1}\wedge \cdots \widehat{e_{n_i}} \cdots \wedge e_{n_j})
\]
The $i^{th}$ Koszul homology is $H_i(\k.)$; Koszul cohomology is
$H^i(Hom_R(\k., R))$.
\end{defn}
The Koszul complex is exact iff $\f$ is a regular sequence. We will
focus on the case where $\f=\{f_0,f_1,f_2\}$ is a basepoint free
subset of $R_{\bf d}$. Hence

\begin{equation}\label{eqKoszul1}
 \k.(\f,R): 0\to R(-3{\bf d})\stackrel{\delta_3}{\lto} R(-2{\bf d})^3\stackrel{\delta_2}{\lto} R(-{\bf d})^3\stackrel{\delta_1}{\lto} R\to 0.
\end{equation}

Let $Z_i$ and $B_{i}$ be the modules of Koszul $i$-th cycles and boundaries, graded so that the inclusion maps $Z_i, B_{i}\subset K_i$ are of degree $(0,0)$, and
let $H_i=Z_i/ B_{i}$ denote the $i$-th Koszul homology module. Since $\delta_1(p_1,p_2,p_3)=\sum_{i=1}^3 p_if_i$,
\[
  H_0=\coker(\delta_1)=R/I_W.
\]
Since $\sqrt{I_W} = B$, the codimension of $I_W$ is two, so since $\f$
has three generators, $\f$ is not a regular sequence, and thus $H_1\neq 0$.
Our assumption~\eqref{eq:B} that $\rad(I_W)=B$ means that $\depth(I_W)=2$, and then $H_2=H_3= 0$. 

From the
definition of Koszul homology, the syzygy module $\Syz(\f):=\ker(\delta_1)$. Since 
$H_1\neq 0$, the map $\delta_2$ in the 
Koszul complex \eqref{eqKoszul1} factors through $\Syz(\f)$ as $R(-2{\bf d})^3\stackrel{\delta_2}{\lto} \Syz(\f)$ but is not surjective. 
The module $\im(\delta_2)$ is called the module of Koszul syzygies. Thus, the size of non-Koszul syzygies is measured by $H_1$.


\subsection{Determining $H_1(\k.(\f,R))$} \label{ssec:H1}

Since $\rad(I_W)=B$, the modules $H_0$ and $H_1$ are supported on $B$. In particular, we have that $H_B^i(H_1)=0$ for $i>0$ and 
hence, $H_B^0(H_1)=H_1$. This says that the Koszul complex \eqref{eqKoszul1} is not acyclic globally, but it is acyclic off 
$V(B)$, i.e.\ that for every prime $\pp\not\subset B$ the localization $(\k.(\f,R))_\pp$ of \eqref{eqKoszul1} at $\pp$ is acyclic.

Consider the extended Koszul complex of \eqref{eqKoszul1}
\begin{equation}\label{eqKoszul2}
 \k.: 0\to R(-3{\bf d})\stackrel{\delta_3}{\lto} R(-2{\bf d})^3\stackrel{\delta_2}{\lto} R(-{\bf d})^3\stackrel{\delta_1}{\lto} R\to R/I_W\to 0.
\end{equation}
For the complex \eqref{eqKoszul2} we have that $H_i= 0$ if $i\neq 1$.  The following theorem charac\-terizes
$H_1$.

\begin{thm}\label{thm:H1}
There is an isomorphism of bigraded $R$-modules
\[
 H_1 \cong \ker\paren{ H^2_B(R(-3{\bf d}))\stackrel{\delta}{\to} \paren{H^2_B(R(-2{\bf d}))}^3}.
\]
\end{thm}
\begin{proof}
 Consider the \v{C}ech-Koszul double complex $\check \Cc^\bullet_B(\k.)$ that is obtained from \eqref{eqKoszul2} 
 by applying the \v{C}ech functor $\check \Cc^\bullet_B(-)$.

Consider the two spectral sequences that arise from the double complex $\check \Cc^\bullet_B (\k.)$. 
We will denote by $ {}_hE$ the spectral sequence that arises taking first homology horizontally, this is, computing first the 
Koszul 
homology, and by $ {}_vE$ the spectral sequence that is obtained by
computing first the \v{C}ech cohomology. The second page of the
spectral sequence of the horizontal filtration is:
\[
  ^2_hE^{ij} = H^i_B (H_j(\k.)).
  \]
Since $H_B^i(H_1)=0$ for $i>0$ and $H_B^0(H_1)=H_1$, we have
\[
\ _h ^2E^{ij} = H^i_B (H_j) = \left\lbrace\begin{array}{ll}
  H_1& \mbox{for }j=1 \mbox{ and } i=0\\
   0 & \mbox{otherwise.} \end{array}\right.
\]
We conclude that
\[
 H^\bullet_B (H_\bullet) \To H_1.
\]

The second spectral sequence has
\[
  _v ^1E^{ij} = H^i_B (K_j),
\]
where $K_i$ is the $i$-th module from the right in Equation~\eqref{eqKoszul2}. Precisely, we have
\[
  \begin{array}{ccc}
    _v ^1E^{i,-1} &= &H^i_B (R/I_W)\\
    _v ^1E^{i,0} &  = &H^i_B (R)\\
    _v ^1E^{i,1}& = &H^i_B (R(-{\bf d})^3)\\
    _v^1E^{i,2}& = &H^i_B (R(-2{\bf d})^3)\\
    _v ^1E^{i,3}& =& H^i_B (R(-3{\bf d}))
  \end{array}
  \]
Therefore, the $^1E$ page of the vertical spectral sequence is
\[
 \xymatrix@1@C=15pt@R=15pt{ 
 {0} \ar[r] & \omega_{R}^\ast(3{\bf d})\ar[r] & \paren{\omega_{R}^\ast(2{\bf d})}^3\ar[r] & 
\paren{\omega_{R}^\ast({\bf d})}^3\ar[r] & \omega_{R}^\ast\ar[r] & H^3_B(R/_WI)  \\
 0 \ar[r] & {H^2_B(R(-3{\bf d}))}\ar[r]^\delta\ar@{-->}[drr]\ar@{..>}[ddrrr] & 
\paren{H^2_B(R(-2{\bf d}))}^3\ar[r] & \paren{H^2_B(R(-{\bf d}))}^3\ar[r] & H^2_B(R)\ar[r] & H^2_B(R/I_W)\\
 0 & 0 & {0} & 0 & 0 & H^1_B(R/I_W)\\
 0 & 0 & 0 & {0} & 0 & H^0_B(R/I_W)
}
\]

By comparing both spectral sequences, we conclude that
\[
 H_1 \cong \ker\paren{ H^2_B(R(-3{\bf d}))\stackrel{\delta}{\to} \paren{H^2_B(R(-2{\bf d}))}^3}.\qedhere
\]
\end{proof}

\begin{cor}\label{corExactSeqRectangle} The sequence 
\[
 0\to R(-3{\bf d})\stackrel{\delta_3}{\lto} R(-2{\bf d})^3\stackrel{\delta_2}{\lto} \Syz(\f) \to H^2_B(R(-3{\bf d}))\stackrel{\delta}{\to} \paren{H^2_B(R(-2{\bf d}))}^3
\]
is exact.
\end{cor}
\begin{proof}
  From equation \eqref{eqKoszul2}, $0\to R(-3{\bf d})\stackrel{\delta_3}{\lto} R(-2{\bf d})^3\stackrel{\delta_2}{\lto} \Syz(\f) \to H_1\to 0$ 
is exact. 
Theorem~\ref{thm:H1} gives $ H_1 \cong \ker\paren{ H^2_B(R(-3{\bf d}))\stackrel{\delta}{\to} \paren{H^2_B(R(-2{\bf d}))}^3}$. The 
result follows by connecting the two sequences.
\end{proof}

\subsection{Understanding $\paren{H_1}_{\a}$}\label{sec:generalcase}

We have the following consequence of Theorem~\ref{thm:H1}.

\begin{cor}
\[
 \supp_{\ZZ^2}(H_1) \subset -\NN\times \NN+(3d_1-2,3d_2)\cup\NN\times -\NN+(3d_1,3d_2-2).
\]
\end{cor}
\begin{proof}
A direct computation using the Mayer-Vietoris sequence yields
\begin{enumerate}
 \item $H^2_B(R)=H^2_{\aaa}(R)\oplus H^2_{\bb}(R)=\paren{\omega_{R_1}^\ast\otimes_{\K} R_2}
\oplus\paren{R_1\otimes_{\K}\omega_{R_2}^\ast}$,
 \item $H^3_B(R)=H^4_{\mm}(R)=\omega_{R}^\ast$,
 \item $H^i_B(R)=0$ for all $i\neq 2,3$, 
\end{enumerate}
where $\omega_S^\ast$ denotes the canonical dualizing module of $S$.
Since we have that the $\supp_{\ZZ^2}(H^2_B(R)) = -\NN\times \NN+(-2,0)\cup\NN\times -\NN+(0,-2)$, by shifting we get that
\[
 \supp_{\ZZ^2}(H_1) \subset -\NN\times \NN+(3d_1-2,3d_2)\cup\NN\times -\NN+(3d_1,3d_2-2). \qedhere
\]
\end{proof}

\begin{cor}\label{cor:H1}
Consider the map $H^2_B(R(-3{\bf d}))\stackrel{\delta}{\to} \paren{H^2_B(R(-2{\bf d}))}^3$. For every $\a=(a_1,a_2)$, we get
\[
 \paren{H_1}_{(a_1,a_2)} \cong \ker\paren{ 
 \begin{matrix}
                                        R_{(3d_1-a_1-2,-3d_2+a_2)}\\ \oplus \\
                                        R_{(-3d_1+a_1,3d_2-a_2-2)}
                                       \end{matrix}
                                       \stackrel{\delta_{a}}{\lto} \paren{ \begin{matrix}
                                       R_{(2d_1-a_1-2,-2d_2+a_2)}\\ \oplus \\
                                        R_{(-2d_1+a_1,2d_2-a_2-2)}
                                        \end{matrix}}^3}
\]
is an isomorphism of $\K$-modules. Identifying the target with
\[R_{(2d_1-a_1-2,-2d_2+a_2)}^3 \oplus R_{(-2d_1+a_1,2d_2-a_2-2)}^3, \mbox{ we have }
\]
\[
\delta_\a=\begin{pmatrix}
                                          \phi_1& 0\\ 0 &\phi_2
                                         \end{pmatrix}, \mbox{ with }
                                         \]
                                         
\begin{eqnarray}\label{def:phi12}
\phi_1:R_{(-3d_1+a_1,3d_2-a_2-2)}\to R_{(-2d_1+a_1,2d_2-a_2-2)}^3, \nonumber \\
\phi_2:R_{(3d_1-a_1-2,-3d_2+a_2)}\to R_{(2d_1-a_1-2,-2d_2+a_2)}^3.
\end{eqnarray}
\end{cor}

For $d_1,d_2\geq 2$, the previous result gives a description of the kernel $\paren{H_1}_{(a_1,a_2)}$:

\begin{footnotesize}
\[\begin{array}{c||c|c|c|c|c}
\left[3d_2,+\infty \right)			& \ker(\phi_2)&R_{(3d_1-a_1-2,-3d_2+a_2)}&0& 0\\  \hline
3d_2-1		&0&0&0& 0 \\  \hline
\left(2d_2-2,3d_2-2\right]		& 0&0&0& R_{(-3d_1+a_1,3d_2-a_2-2)} \\\hline
\left(-\infty,2d_2-2\right]	& 0 & 0 &0& \ker(\phi_1) \\ \hline\hline
		& \left(-\infty,2d_1-2\right] &\left(2d_1-2,3d_1-2\right]&3d_1-1	& 
\left[3d_1,+\infty \right)  
\end{array}\]
\end{footnotesize}
\medskip

\noindent The next examples illustrate the map $\delta_\a$ of Corollary~\ref{cor:H1} in a particular case in which
the three polynomials $f_i$ can be factored as two polynomials with bidegrees $(1,0)$ and $(0,n)$.

\begin{example0}\label{ex:maps}
Let $\d=(1,n)$, $f_0=su^n$, $f_1=tv^n$, $f_2=(s+t)(u^n+v^n)$, and $(a_1,a_2)=(3,n)$. Then 
\begin{equation}\label{delta}
 \paren{H_1}_{(3,n)} 
 = \ker\paren{ (R_1\otimes_{\K}\omega_{R_2}^\ast)_{(0,-2n+2)} \stackrel{\delta_{(3,n)}}{\lto} 
\paren{ (R_1\otimes_{\K}\omega_{R_2}^\ast)_{(1,-n+2)}}^3}
\end{equation}
and with the standard identification of the canonical dualizing modules, one has
\[
 \paren{H_1}_{(3,n)} 
 \cong \ker\paren{ R_{(0,2n-2)} \lto \paren{ R_{(1,n-2)}}^3}
\]
The map $\delta_{(3,n)}$ in Equation \eqref{delta} is given by multiplication by $f_i$. 
Precisely, given $a\geq 0$, $\delta_{(3,n)}$ is as follows 
\[
\frac{1}{uv}\frac{1}{u^av^{2n-2-a}} \mapsto \frac{1}{uv}\left( \frac{1}{u^av^{2n-2-a}}f_0, 
\frac{1}{u^av^{2n-2-a}}f_1, \frac{1}{u^av^{2n-2-a}}f_2 \right).
\]

Thus, fixing a basis $\mathcal B$ for
$(R_1\otimes_{\K}\omega_{R_2}^\ast)_{(0,-2n+2)}$ and also fixing a basis $\mathcal B'$ for
$(R_1 \otimes_{\K}\omega_{R_2}^\ast)_{(1,-n+2)}^3$, 
$|\delta_{(3,n)}|_{{\mathcal B}{\mathcal B}'}$ is a $(3\cdot 2 (n-1))\times (2n-1)$-matrix given by the coefficients $coef_{\mathcal B'}((f_0,f_1,f_2)\cdot {\mathcal B}_i)$ of the 
$i$-th element of $\mathcal B$ multiplied by one of the $f_j$ ($j$ depending on the row), written in the basis $\mathcal B'$. 
\end{example0}
 \noindent We now exhibit the matrices in Example \ref{ex:maps} in bidegree $(1,6)$.
\begin{example0}\label{ex:maps6}
Set for instance $n=6$ (so $\d=(1,6)$), $|\delta_{(3,n)}|_{{\mathcal B}{\mathcal B}'}$ is a $(3\cdot 10)\times 11$-matrix. One can take
${\mathcal B}=\left\lbrace \frac{1}{uv}\frac{1}{u^6},\hdots,\frac{1}{uv}\frac{1}{v^6} \right\rbrace $ 
and 

${\mathcal B}'=\left\lbrace\frac{1}{uv}\frac{s}{u^4},\hdots,\frac{1}{uv}\frac{s}{v^4},\frac{1}{uv}\frac{t}{u^4},
\hdots\frac{1}{uv }\frac{t}{ v^4} \right\rbrace \times \{(1,0,0),(0,1,0),(0,0,1)\}.$

In this case, one has that the $10$-tuple, corresponding to the `upper third' of the first
 column of  $|\delta_{(3,n)}|_{{\mathcal B}{\mathcal B}'}$, induced by multiplying by $f_0$ is 
\[
coef_{\mathcal B'}(f_0\cdot {\mathcal B}_1)=coef_{\mathcal B'}\left(f_0\cdot \frac{1}{uv}\frac{1}{u^{10}}\right)=
coef_{\mathcal B'}\left(\frac{1}{uv}\frac{s}{u^{4}}\right)=(1,0,0,\hdots,0).
\]
 And, because of the structure of multiplication on $\omega_{R_2}^\ast$, it is easy to see that in 
$(R_1\otimes_{\K}\omega_{R_2}^\ast)_{(1,-n+2)}$,  $f_j\cdot u^{-6}v^{-6}=0$. 
Thus, the 6th column of $|\delta_{(3,n)}|_{{\mathcal B}{\mathcal B}'}$ is 
zero, and the rest are not.

The matrix $|\delta_{(3,n)}|_{{\mathcal B}{\mathcal B}'}$ has the following shape
\[
\left(\begin{array}{c|c|c}
Id_{5\times 5} & 0 & 0 \\
0 & 0 & 0 \\ \hline
0 & 0 & 0 \\ 
0 & 0 & Id_{5\times 5} \\ \hline
Id_{5\times 5}& 0 & Id_{5\times 5} \\
Id_{5\times 5} & 0 & Id_{5\times 5}
\end{array}\right),
\]
where $Id_{5\times 5}$ is the ${5\times 5}$-identity matrix.

Finally, we conclude that $\corank(|\delta_{(3,n)}|_{{\mathcal B}{\mathcal B}'})=1$. 
Since the matrix above induces a morphism from $k^{11}\to  k^{30}$, we have that 
\[
 HF_{H_1}(3,n)= \dim(\ker(|\delta_{(3,n)}|_{BB'}))=\corank(|\delta_{(3,n)}|_{{\mathcal B}{\mathcal B}'})=1.
\]
This in particular says that there is only one non-Koszul syzygy
spanning every other non-Koszul syzygy.
\end{example0}
\noindent Examples~\ref{ex:maps} and \ref{ex:maps6}  are explained by Theorem~\ref{smoothconicres}.

\subsection{The generic case}\label{ssec:gencase}

We give the definition of {\em generic} bihomogeneous polynomials $f_i$ of the same bidegree ${\bf d} \in \ZZ^2_{>0}$. Note
that our assumption~\eqref{eq:B} implies that no $d_i$ could be equal to $0$.

\begin{defn}\label{def:hda}
Given $\d \in \ZZ^2_{>0}$ and $\bf f= \{f_0, f_1, f_2\}$ of bidegree $\d$ satisfying~\eqref{eq:B}, we say that
$\f$ is generic if the maps $\phi_1$ and $\phi_2$ in~\eqref{def:phi12} in Corollary~\ref{cor:H1} have full rank, for any $\a\in \ZZ_{\geq 0}^2$.
\end{defn}
This condition of full rank is not true only under the assumption~\eqref{eq:B}.
For instance, the factorizable polynomials in Example~\ref{ex:maps6} are not generic, but our Conjecture~\ref{conj:ANH} below states that
the maps have full rank for polynomials with generic coefficients. 

If we denote by $n_\d(\a)$ the difference of dimensions:
\[
  n_d(\a):=\dim_{\K}\paren{ 
 \begin{matrix}
                                        R_{(3d_1-a_1-2,-3d_2+a_2)}\\ \oplus \\
                                        R_{(-3d_1+a_1,3d_2-a_2-2)}
                                       \end{matrix}} -    3\dim_{\K} \paren{ \begin{matrix}
                                       R_{(2d_1-a_1-2,-2d_2+a_2)}\\ \oplus \\
                                        R_{(-2d_1+a_1,2d_2-a_2-2)}
                                        \end{matrix}}\in \ZZ[a,d],
\]
we have by Corollary~\ref{cor:H1} that
\[
 \dim_{\K} \paren{H_1}_{(a_1,a_2)} \geq n_\d(\a).
 \]

\medskip


For any real number $c$, denote 
\begin{equation}\label{eq:c}
c_+=\max(c,0), \quad  c_-=\max(0,-c).
\end{equation}
Note that for any $c \in \mathbb R$,
$ c = c_+ - c_-$
and only one of these two numbers can be positive.

Given $\d,\a\in \Z^2 _{\ge 0}$, we set
\[\dom_d(\a):=(0,-3d_1+a_1+1)_+ (0,3d_2-a_2-1)_+ +(0,3d_1-a_1-1)_+ (0, -3d_2+a_2+1)_+\]
\[\cod_d(\a):=(0,-2d_1+a_1+1)_+ (0,2d_2-a_2-1)_+ + (0,2d_1-a_1-1)_+ (0,-2d_2+a_2+1)_+.\]

\medskip

The following Lemma is straigthforward taking into account that  a linear
map $V_1\to V_2$ between two $\K$-vector spaces of finite dimension  is of maximal rank if and only if
the dimension of its kernel of equals $(\dim_{\K}(V_1)-\dim_{\K}(V_2))_+$. 

\begin{lemma} \label{lem:generic}
Let $\bf f= \{f_0, f_1, f_2\}$ of bidegree $\d$ satisfying~\eqref{eq:B}. Then, $\f$ is generic if and only if
\begin{equation}\label{eq:nda}
n_{\d}(\a)\, = \, (\dom_d(\a)- 3 \cod_d(\a))_+.
\end{equation}
In this case,  for any $\a \in \Z_{\ge 0}^2$ we have the equality:
\begin{equation}\label{eq:nda2}
 \dim_{\K} \paren{H_1}_{(a_1,a_2)} = n_\d(\a).
 \end{equation}
\end{lemma}

In fact, we conjecture that this is indeed the generic behavior

\begin{conj}  \label{conj:ANH}  There exists a nonempty open set in the space of coefficients of the
polynomials $f_i$ where $\f$ is generic
according to Definition~\ref{def:hda} and hence
 $\dim \paren{H_1}_{{\bf a}}=n_{\bf d}({\bf a})$ for any $\a \in \ZZ^2_{\ge 0}$  by Lemma~\ref{lem:generic}. 
\end{conj}

\begin{remark}\label{rem:H1}
Note that Corollary~\ref{cor:H1} proves that Conjecture~\ref{conj:ANH} is always true for polynomials
$f_i$ satisfying assumption~\eqref{eq:B} outside of the range where we
have $(a_1\ge 3d_1 \text{ and } d_2 \le a_2\le 2 d_2-2)$
and  $(a_2\ge 3d_2 \text{ and } d_1 \le a_1\le 2 d_1-2)$.
\end{remark}

\subsection{The Fr\"oberg-Lindqvist conjecture on bigraded Hilbert series}\label{sec:R}

For any bidegree $\a$, we denote by $\chi_{\k.}(\a)$ the Euler characteristic of the $\ell$-strand of the
Koszul complex~\eqref{eqKoszul1} and let $S(x,y) = \sum_{\a \in \Z_{\ge 0}^2} \chi_{\k.}(\a) x^{a_1} y^{a_2}$.
Then,
\begin{equation}\label{eq:S}
S(x,y) \, = \, \frac{(1-x^{d_1}y^{d_2})^3}{(1-x)^2(1-y)^2}.
\end{equation}
We denote by $S(x,y)_+$ the series suppported in $\Z_{\ge 0}^2$ with coefficients $\chi_{\k.}(\ell)_+$.
The following lemmas are straightforward.

\begin{lem0} \label{lem:S+}
$S(x,y)_+$ and $S(x,y)_-$ are also rational functions of $x,y$.
\end{lem0}

\begin{lem} \label{lem:chik}
Define regions 
\[
\begin{array}{ccc}
A_1 & = & a_1 < d_1  \mbox{ or } a_2 < d_2\\
A_2 & = &  (d_1 \le a_1< 2d_1 \mbox{ and } d_2 \le a_2)  \mbox{ or }   (d_1 \le a_1 \mbox{ and  } d_2 \le a_2 < 2d_2)\\
A_3 & = &  (2d_1 \le a_1< 3d_1 \mbox{ and } 2d_2 \le a_2) \mbox{ or }  (a_1 < 3 d_1 \mbox{ and } 2d_2 \le a_2 < 3 d_2)\\
A_4 & = & 3d_1 \le a_1  \mbox{ and } 3d_2  \le a_2
\end{array}
\]
For any $\a \in {\mathbb Z}_{\ge 0}^2$, the coefficient $\chi_{\k.}(\a)$ equals the following:
\begin{footnotesize}
\[
\begin{array}{ccc}
(a_1+1) (a_2+1) & \mbox{ in } & A_1\\
(a_1+1) (a_2+1)  - 3 (a_1-d_1+1)(a_2-d_2+1)& \mbox{ in } & A_2\\
(a_1+1) (a_2+1) - 3 (a_1-d_1+1)(a_2-d_2+1) +  3 (a_1- 2 d_1+1)(a_2- 2 d_2+1)& \mbox{ in } & A_3\\
0 & \mbox{ in } & A_4
\end{array}
\]
\end{footnotesize}
\end{lem}

The following table shows in which bidegrees the Euler characteristic $\chi_{\k.}(\a)$  is posi\-ti\-ve, negative or zero.
\begin{footnotesize}
\[\begin{array}{c||c|c|c|c|c}
\left[3d_2,+\infty \right)			& + & +/hiperb/-& - & 0 & 0 \\  \hline
3d_2-1			& + & + & 0 & 0 & 0 \\  \hline
\left(2d_2-2,3d_2-2\right]			& + & + & + & 0 & - \\  \hline
\left(d_2-1,2d_2-2\right]			& + & + & + & + &  +/hiperb/-  \\  \hline
\left(0,d_2-1\right]			& + & + & + & + & + \\  \hline\hline
		& \left(0,d_1-1\right] & \left(d_1-1,2d_1-2\right] &\left(2d_1-2,3d_1-2\right]&3d_1-1	& 
\left[3d_1,+\infty \right) 
\end{array}\]
\end{footnotesize}

\noindent Here the notation $+/hiperb/- $ means that there is a hyperbola where $\chi_{\k.}$ vanishes, separating the positive from the negative values.
Here is an equivalent version of the Fr\"oberg and Lundqvist Conjecture~8 from~\cite{FL}:

\begin{conj}\label{conj:RS} There exists a nonempty open set in the space of coefficients of the
polynomials $f_i$ for which $\dim(R/I_W)_\a = \chi_{\k.}(\a)_+$ for any $\a \in \ZZ^2_{\ge 0}$.
\end{conj}

\noindent We now prove that Conjecture~\ref{conj:ANH} and~\ref{conj:RS} are indeed equivalent.

\begin{lem}\label{lem:H0-H1}
Let $\f=\{f_0, f_1, f_2\}$ be bihomogeneous polynomials of bidegree $\d \in \Z^2_{>0}$ satisfying~\eqref{eq:B}.
We have the equality:
\begin{equation}\label{eq:H01}
\chi_{\k.}(\a) \, = \dim (R/I_W)_\a  \, - \, \dim \paren{H_1}_{\a}.
\end{equation}
\end{lem}

As we remarked in Section~\ref{ssec:K}, our assumption~\eqref{eq:B}  implies that $H_2=H_3= 0$, and so the
proof of Lemma~\ref{lem:H0-H1} is immediate.

\medskip

\noindent We compare the conjectural dimension $n_\d(\a)$ of ${H_1}_{\a}$ with the coefficients of $S$.

\begin{lem}\label{lem:nd}
For any $\a \in \Z^2_{\ge 0}$ we have the equality:
\begin{equation}
\dim n_\d(\a) = \chi_{\k.}(\a)_-.
\end{equation}
\end{lem}

\begin{proof}
By Lemma~\ref{lem:chik} are sixteen domains of polynomiality of $\chi_{\k.}$.
Consider for instance the case $a_1>3d_1-1, 2d_2-1 < a_2 \le 3d_2-1$. Then,
$n_\d(\a)=(3d_1-a_1-1)(3d_2-a_2-1)$, while $\chi_{\k.}(\a)_-=
(a_1+1)(a_2+1)-3(-d_1+a_1+1)(-d_2+a_2+1)+3(-2d_1+a_1+1)(-2d_2+a_2+1)$
and it is a simple computation to check that they coincide.
The other cases are similar.
\end{proof}
By Lemma~\ref{lem:S+} the generating series $T(x,y) = \sum_\a \dim \paren{H_1}_{\a} x^{a_1} y^{a_2}$ is 
a rational function. Hence

\begin{prop}\label{prop:conjs}
Conjectures~\ref{conj:ANH} and~\ref{conj:RS} are equivalent.
\end{prop}

\begin{proof}
Assume $\dim \paren{H_1}_{\a} = n_\d(\a)$. Using
 Lemma~\ref{lem:nd}, we substitute this value in the statement of Lemma~\ref{lem:H0-H1}:
\[\chi_{\k.}(\a) = \chi_{\k.}(\a)_+ - \chi_{\k.}(\a) _- = \dim (R/I_W)_\a  \, - \, \chi_{\k.}(\a)_-,\]
which says that $\dim (R/I_W)_\a = \chi_{\k.}(\a)_+$. The converse is similar.
\end{proof}

\begin{remark}\label{rem:H1bis}
By Remark~\ref{rem:H1}, we have that $\dim \paren{H_1}_{\a} = n_\d(\a)$
 is true for polynomials
$f_i$ satisfying assumption~\eqref{eq:B} except in the ranges $(a_1\ge 3d_1 \text{ and } d_2 \le a_2\le 2 d_2-2)$
and  $(a_2\ge 3d_2 \text{ and } d_1 \le a_1\le 2 d_1-2)$. So, we deduce from the proof of Proposition~\ref{prop:conjs} that Conjecture~\ref{conj:RS} is true outside these ranges.
\end{remark}

\noindent We end this section with an easy corollary.

\begin{cor}
If Conjectures~\ref{conj:ANH} and~\ref{conj:RS} hold, then for any $\f$ regular 
either ${R/I_W}_\a=0$ or ${H_1}_{\a}=0$  for any $\a \in {\mathbb Z}_{\ge 0}^2$. 
\end{cor}

\begin{proof}
If the conjectures are valid for any $\f$ regular, we have for any bidegree $\a$ that $\dim(R/I_W)_\a=\chi_{\k.}(\a)_+$ and $\dim \paren{H_1}_{\a}=\chi_{\k.}(\a)_-$, and as we remarked after~\eqref{eq:c}, at most one of these numbers can be nonzero.
\end{proof}


\section{Koszul homology $H_1(\k.(\f,R))$ for ${\bf d}=(1,n)$}\label{sec:1n}
From now on, we specialize our study to bidegrees of the form ${\bf d}=(1,n)$, always assuming that $W$ is basepoint free. 
This case is both the natural sequel to the study of the $(1,2)$ case studied in \cite{cds}, as well as a key ingredient for better understanding the general
case. It splits the analysis into separate parts, in a way that we make precise below. Theorem~\ref{th:KoszulM}  relates the Betti numbers $\beta_{1, {\bf a}}$ with the Koszul homology of $H_1$ with respect to the sequence $\{s,t,u,v\}$,  for any bidegree $\d$.

For degree $(1,n)$, Corollary~\ref{cor:H1} yields the following description of  $\paren{H_1}_{(a_1,a_2)}$:
\[
 \paren{H_1}_{(a_1,a_2)} \cong \ker\paren{ 
 \begin{matrix}
                                        R_{(1-a_1,-3n+a_2)}\\ \oplus \\
                                        R_{(-3+a_1,3n-a_2-2)}
                                       \end{matrix}
                                       \stackrel{\delta_{(a_1,a_2)}}{\lto} \paren{ \begin{matrix}
                                       R_{(-a_1,-2n+a_2)}\\ \oplus \\
                                        R_{(-2+a_1,2n-a_2-2)}
                                        \end{matrix}}^3}.
\]
 The table given in Section \ref{sec:generalcase} reduces to
\[\begin{array}{c||c|c|c|c}
\left[3n,+\infty \right)			&R_{(1-a_1,-3n+a_2)}& 0& 0\\  \hline
3n-1		&0&0& 0 \\  \hline
\left(2n-2,3n-2\right]		&0&0& R_{(-3+a_1,3n-a_2-2)} \\\hline
\left(-\infty,2n-2\right]		&0&0& \ker(\phi_1) \\ \hline\hline
		&1 &2	& \left[3,+\infty \right)  
\end{array},
\]
with 
\[
\phi_1:R_{(-3+a_1,3n-a_2-2)}\to R_{(-2+a_1,2n-a_2-2)}^3.
\]
So the region where interesting behavior occurs is in multidegree $(a_1,a_2)$, with 
\[
\begin{array}{ccc}
a_2 \ge 3n & \mbox{ and } & a_1 =1\\
           & \mbox{ or } &               \\
a_1 \ge 3 & \mbox{ and }&  a_2 \le 2n-2
\end{array}
\]
Since we need $I_W$ to be nonzero, we have 
the constraint that  $ a_1 \ge d_1, a_2 \ge d_2$, so for ${\bf d} = (1,n)$, the only region of interest is bidegree $(a_1,a_2)$, with 
\[
a_1 \ge 3 \quad \mbox{ and } \quad 2n-2 \ge a_2 \ge n,
\]
corresponding to $\ker(\phi_1)$ defined in \ref{sec:generalcase}. We study $n \ge 3$; $n=2$ is  analyzed in \cite{cds}. 
\subsection{Tautological first syzygies: degrees $(1,*)$ and $(2,*)$}\label{TautSyz}
\begin{lemma}\label{Alicia}
There is a unique minimal first syzygy on $I_W$ in bidegree $(1,3n)$.
\end{lemma}
\begin{proof}
Because $\phi_2: \K \to 0$, $\ker(\phi_2) \simeq \K$, and we can describe the syzygy explicitly as follows (it is of the type appearing in Lemma 6.1 of \cite{cds}.)
Write 
\[
\begin{array}{ccc}
f_0 &=& s \cdot p_0+t\cdot q_0\\
f_1 &=& s \cdot p_1+t \cdot q_1\\
f_2 &=& s \cdot p_2 +t\cdot q_2,
\end{array}
\]
with the $p_i ,q_i\in k[u,v]_n$. Then 
\[
\det \left[ \!
\begin{array}{ccc}
f_0 & p_0 &q_0\\
f_1& p_1 &q_1\\
f_2 & p_2 &q_2 
\end{array}\! \right] =0, 
\]
so the $2 \times 2$ minors in $q_i$ and $p_i$ give a syzygy with entries of bidegree 
$(0,2n)$, hence of bidegree $(1,3n)$ on $I_W$.  It is minimal since any 
syzygy of lower degree would be of the form $(0,d)$ with $d <2n$. This 
would force $W$ to have basepoints; to see this note that a syzygy 
$(s_0,s_1,s_2)$ of bidegree $(0,d)$ must be in the kernel of the map 
\[
\theta = \left[ \!
\begin{array}{ccc}
p_0 &p_1 &p_2\\
q_0& q_1 &q_2
\end{array}\! \right],
\] 
by splitting out the $s$ and $t$ components. But $\theta$ gives a map 
$\Oc_{\P^1}^3(-n) \rightarrow \Oc_{\P^1}^2$ with zero cokernel, because the 
rank of $\theta$ drops on the locus of the $2 \times 2$ minors of 
$\theta$; such a point would be a basepoint of $W$. A Chern class 
computation shows $\ker(\theta) \simeq \OPone(-3n)$; in fact it 
consists of the $2 \times 2$ minors of $\theta$. 
\end{proof}

\begin{prop}\label{prop3.2}
The syzygy of Lemma~\ref{Alicia} and the three Koszul syzygies
generate a pair of minimal second syzygies of bidegree $(2,3n)$. 
Furthermore, there is a minimal third syzygy of bidegree $(3,3n)$. 
\end{prop}
\begin{proof}
Let
\[
\begin{array}{ccc}
f_0 &=& s \cdot p_0+t\cdot q_0\\
f_1 &=& s \cdot p_1+t \cdot q_1\\
f_2 &=& s \cdot p_2 +t\cdot q_2,
\end{array}
\]
with the $a_i ,b_i\in k[u,v]_n$, and consider the
submatrix $A$ of $\partial_1$ generated by the syzygy of Lemma~\ref{Alicia}
and the three Koszul syzygies:
\[
\left[ \!
\begin{array}{ccccc}
q_1p_2-p_1q_2    &t \cdot q_1+s \cdot p_1   &  t \cdot q_2 +s\cdot p_2    & 0                       \\
p_0q_2-q_0p_2    &-(t \cdot q_0+s\cdot p_0)  & 0                          &  t \cdot q_2 +s\cdot p_2 \\
q_0p_1-p_0q_1    & 0                         &-(t \cdot q_0+s\cdot p_0)     & -(t \cdot q_1+s \cdot p_1) 
\end{array}\! \right] 
\]
The columns of the matrix $A'$
\[
\left[ \!
\begin{array}{ccc}
s    &t  & 0\\
q_2 &-p_2  & f_2 \\
-q_1  &p_1  & -f_1  \\
q_0 & -p_0  & f_0  
 \end{array}\! \right] 
\]
are in the kernel of $A$; the rightmost column is the second Koszul
syzygy on $I_W$. As $[t,-s,1]$ is in the kernel of $A'$, we see that
there is a third syzygy of bidegree $(3,3n)$. Note that by
Theorem~\ref{1syz3m} the second Koszul syzygy is not minimal, but can
be represented in terms of the syzygies appearing in
Theorem~\ref{1syz3m}.
\end{proof}

The results of Section \ref{sec:generalcase} show that there are no minimal first syzygies in bidegree $(2, *)$ except for 
the Koszul syzygies in degree $(2,2n)$.  This can be seen explicitly, as follows. First, a syzygy with entries of bidegree $(1,m)$ satisfies
\[
(sg_0 +th_0)f_0+(sg_1 +th_1)f_1+(sg_2 +th_2)f_2=0,
\]
with the $f_i$ as in the previous lemma. 
Note that $\langle p_0,p_1,p_2\rangle$ and $\langle q_0,q_1,q_2\rangle$
are both basepoint free on $\P^1$; for otherwise vanishing of 
$\{ p_0,p_1,p_2, t \}$ would give a basepoint on $\PP$ and also for
for  $\{ q_0,q_1,q_2, s \}$. If $(u_0:v_0) \in \P^1$ is a point where the 
rank of 
\[
\left[ \!
\begin{array}{ccc}
p_0 & p_1  &p_2\\
q_0 & q_1  &q_2
\end{array}\! \right]
\]
is one, then $s=u_0,t=v_0$ is a basepoint of $I_W$. 
Using that $f_i=sp_i+tq_{i}$,  multiplying out and collecting
the coefficients of the $\{s^2, st, t^2\}$ terms shows that
$[g_0,g_1,g_2,h_0,h_1,h_2]$ is in the kernel of the matrix
\[
M=\left[ \!
\begin{array}{cccccc}
p_0 & p_1  &p_2 &0     &0      &0\\
q_0 & q_1  &q_2  &p_0 & p_1 &p_2 \\
0     & 0     &0    & q_0 & q_1 &q_2
\end{array}\! \right].
\]
The remarks above show that the kernel is free of rank three, with
first Chern class $6n$. The matrix $K$ below satisfies these properties
and clearly $MK=0$:
\[
\left[ \!
\begin{array}{ccc}
-p_1 & 0  &p_2\\
p_0 & -p_2 & 0 \\
0     & p_1 &p_0\\
-q_1  &0    &-q_2\\
q_0   &-q_2 & 0 \\
0    & q_1  &q_0
\end{array}\! \right].
\]
By the Buchsbaum-Eisenbud criterion, $K = \ker(M)$. But $K$ consists
of exactly the Koszul syzygies. From this it follows that the lowest possible nonzero multidegree in the degree $(1,0)$ 
variables for a non-tautological first syzygy is $(3,m+n) = (2, m) + (1,n)$, with $m \ge 0$.
We tackle this next.


\subsection{First syzygies of degree $(3,*)$}
The next theorem gives a complete description of the first syzygies with entries of degree $(2,m)$, hence which are of total degree $(3,m+n)$.
\begin{thm}\label{threestarsyz}
    For the first Betti numbers,
    \[
      \beta_{1,(3,*)} \in \{1, \ldots,5\}
    \]
 and all possible values between one and five occur.
  \end{thm}
    \begin{proof}
Theorems~\ref{smoothconic}, \ref{smoothconicres}, 
\ref{3pointintersection} treat the case where $W$ meets the Segre
variety $\Sigma_{1,n}$ in a smooth conic $C$ or 3 noncollinear
points $Z$. In these situations we may choose a basis so $W = \Span\{g_0h_0, g_1h_1,g_2h_2\}$ with $g_i$
degree $(1,0)$ and $h_i$ degree $(0,n)$. Theorems~\ref{smoothconicres}
and \ref{3pointintersection} give explicit resolutions for $I_W$ in
these cases. 
\begin{itemize}
  \item When $W\cap \Sigma_{1,n}=C$ there is a single syzygy of 
    bidegree $(3,n)$.
    \item $W\cap \Sigma_{1,n}=Z$ there are two syzygies of bidegrees
      $(3,n+\mu), (3,2n-\mu)$.
      \item The remaining cases are covered in Theorem~\ref{1syz3m} below.
      \end{itemize}
    \end{proof}
    
\begin{thm}\label{1syz3m}
Suppose $\{f_0,f_1,f_2\}=(s q_0 +t q_3, s q_1 +tq_4, s q_2
+tq_5)$,
with the $q_i$ linearly independent $($so $n\geq 5).$
Then there are exactly five minimal first syzygies whose entries are 
quadratic in $\{s,t\}$, obtained from the Hilbert-Burch 
matrix $($\cite{GeomSyz}, Theorem 3.2$)$ $N$ for the ideal $Q=\langle q_0,\ldots, q_5\rangle$. If the columns of $N$ have degrees $\{b_1,\ldots, b_5\}$,
then the syzygies on $I_W$ are of degree $\{(3,2n-b_1),\ldots, (3,2n-b_5)\}$.
\end{thm}
\begin{proof} 
A syzygy with entries of bidegree $(2,m)$ satisifies
\[
(s^2a_0 +sta_1+t^2a_2)f_0+(s^2b_0 +stb_1+t^2b_2)f_1+(s^2c_0
+stc_1+t^2c_2)f_2=0,
\]
so using that $f_i=sq_i+tq_{i+3}$, multiplying out and collecting
the coefficients of the $\{s^3, s^2t, st^2,t^3\}$ terms shows that
$[a_0,b_0,c_0,a_1,b_1,c_1,a_2,b_2,c_2]$ is in the kernel of the matrix
\[
M=\left[ \!
\begin{array}{ccccccccc}
q_0 & q_1  &q_2 &0     &0      &0  &0  &0  &0\\
q_3 & q_4  &q_5  &q_0 & q_1 &q_2  &0  &0  &0\\
0     & 0     &0    & q_3 & q_4 &q_5 &q_0 & q_1 &q_2 \\
0      &0     &0     &0     & 0     &0    & q_3 & q_4 &q_5 
\end{array}\! \right]. 
\]
Since $W$ is basepoint free, it follows that as a sheaf, the
cokernel of 
\[
\Oc^9_{\P^1}(-n) \stackrel{M}{\longrightarrow} \Oc^4_{\P^1}
\]
is zero, hence the kernel of $M$ is a rank five free module 
with first Chern class $9n$. If $N$ denotes the Hilbert-Burch 
matrix of $Q$ then $N$ is a $6 \times 5$ matrix whose maximal minors are $Q$. Write
$n^{ij}_{k}$ for the determinant of the submatrix of $N$
obtained by omitting rows $i,j$ and column $k$ (convention-indexing
starts with 0), and consider the matrix $K$
\[
\left[ \!
\begin{array}{ccccccccc}
-n^{12}_0 & -n^{02}_0  &-n^{01}_0 &n^{15}_0 -n^{24}_0     &n^{05}_0 -n^{23}_0      &n^{04}_0 -n^{13}_0  &-n^{45}_0  &-n^{35}_0  &n^{34}_0\\
-n^{12}_1 & -n^{02}_1  &-n^{01}_1 &n^{15}_1 -n^{24}_1     &n^{05}_1 -n^{23}_1      &n^{04}_1 -n^{13}_1  &-n^{45}_1  &-n^{35}_1  &n^{34}_1\\
-n^{12}_2 & -n^{02}_2  &-n^{01}_2 &n^{15}_2 -n^{24}_2     &n^{05}_2 -n^{23}_2      &n^{04}_2 -n^{13}_2  &-n^{45}_2  &-n^{35}_2  &n^{34}_2\\
-n^{12}_3 & -n^{02}_3  &-n^{01}_3 &n^{15}_3 -n^{24}_3     &n^{05}_3 -n^{23}_3      &n^{04}_3 -n^{13}_3  &-n^{45}_3  &-n^{35}_3  &n^{34}_3\\
-n^{12}_4 & -n^{02}_4  &-n^{01}_4 &n^{15}_4 -n^{24}_4     &n^{05}_4 -n^{23}_4      &n^{04}_4 -n^{13}_4  &-n^{45}_4  &-n^{35}_4  &n^{34}_4
\end{array}\! \right].
\]
Entries of the $i^{th}$ row of $K$ correspond to combinations of certain $4 \times 4$
minors of the submatrix $N_i$ obtained by deleting the $i^{th}$ column
of $N$. 
A computation shows that $M\cdot K^t$=0 and therefore
\[
0 \longrightarrow \bigoplus\limits_{i=1}^5\Oc_{\P^1}(-2n+b_i) \stackrel{K^t}{\longrightarrow}\Oc^9_{\P^1}(-n) \stackrel{M}{\longrightarrow}
\Oc^4_{\P^1}\longrightarrow 0
\]
is exact, by the Buchsbaum-Eisenbud criterion.
\end{proof}
\begin{remark}
If the $q_i$ are not linearly independent, then the basepoint free 
assumption means they span a space of dimension 5 or 4, or fall under 
Theorems~\ref{smoothconicres},  \ref{3pointintersection}. When $\dim 
\Span\{q_0,\ldots,q_5\} \in \{4,5\}$, the matrix $N$ is $5 \times 4$ or $4 \times 3$ and the argument 
of Theorem~\ref{1syz3m} works with appropriate modifications, which we leave to the interested reader. 
\end{remark}

\pagebreak
\begin{cor}\label{123ID}
The tautological syzygies constructed in \S \ref{TautSyz} and the
syzygies of Theorem~\ref{1syz3m} are independent.
\end{cor}
\begin{proof}
The syzygies constructed in Theorem~\ref{1syz3m} cannot be in the span
of the tautological syzygies of \S \ref{TautSyz} because their degree in the $\{u,v\}$
variables is lower than that of the tautological syzygies. On the
other hand, the tautological syzygies cannot be in the span of the syzygies of
Theorem~\ref{1syz3m}, as the tautological syzygies have lower degree
in the $\{s,t\}$ variables.
\end{proof}

\subsection{Computing first Betti numbers, the general setting}\label{ssec:Bettigen}

The Koszul homology of the module $H_1$ is computed from the complex $\m_\bullet:=\m_\bullet((s,t,u,v),H_1)$:
\begin{equation}\label{eqKoszulH1}
 \m_\bullet: 0\to H_1(-4)\stackrel{\varphi_4}{\lto} H_1(-3)^4\stackrel{\varphi_3}{\lto} H_1(-2)^6\stackrel{\varphi_2}{\lto} H_1^4 
\stackrel{\varphi_1}{\lto} H_1 \to 0.
\end{equation}

The bigraded complex $\m_\bullet$ has the following shape:

\begin{small}
\begin{table}[ht]
\begin{center}
\begin{supertabular}{|c|}
\hline 
  $0 \lto H_1(-2,-2) \lto 
 \begin{array}{c}
 H_1(-2,-1)^2\\
 \oplus \\
H_1(-1,-2)^2\\
 \end{array} \lto 
 \begin{array}{c}
 H_1(-2,0)\\
 \oplus \\
H_1(-1,-1)^4\\
\oplus \\
H_1(0,-2)\\
\end{array} \lto 
\begin{array}{c}
 H_1(0,-1)^2\\
 \oplus \\
H_1(-1,0)^2\\
 \end{array} \lto 
H_1\to 0$ \\
\hline 
\end{supertabular}
\end{center}
\label{T3}
\end{table}
\end{small}
\noindent We denote by $H({\mathcal M})_i$ the $i$-th homology module.

\begin{thm0}\label{th:KoszulM}
 For any $\a \in \Z^2_{\ge 0}$, we have the equality
 \begin{equation}\label{eq:betaM}
  \beta_{1,\a} \, = \dim_{\K}(H({\mathcal M})_{1,\a})- \dim_{\K}(H({\mathcal M})_{2,\a}).
 \end{equation}
\end{thm0}

\begin{proof}
First, observe that $ \beta_{1,\a} \neq 0$ iff $(H_1)_\a$ is not spanned by the images of $(H_1)_{\a-e_1}$ and $(H_1)_{\a-e_2}$. Hence $\beta_{1,{\bf a}} \, = \dim_{\K}(H({\mathcal M})_{0,\a})$. As $\dim_{\K}(H({\mathcal M})_{0,\a})$ is the alternating sum of $\dim_{\K}(H({\mathcal M})_{i,\a})$ for $i>0$, it suffices to show that 
\[
H({\mathcal M})_3 = 0 = H({\mathcal M})_4.
\]
Let $K_{i,m}=\Kc_i((s,t,u,v); R)$ and $K_{i,f}=\Kc_i(f; R)$  denote the Koszul complex of $(s,t,u,v)$ and the Koszul complex of $f$ on $R$ respectively. We consider the two spectral sequences coming from the double complex $\Cc_{i,j}=K_{i,m}\otimes_R K_{i,f}$.
\[
^hE_{i,j}^1 = H_i(K_{\bullet,m})\otimes_R K_{i,f}
\]
\[
^vE_{i,j}^1 = K_{i,m}\otimes_R H_j(K_{\bullet,f})
\]

Since $K_{\bullet,m}$ is acyclic, $H_i(K_{\bullet,m})=0$ iff $i\neq 0$, and $H_0(K_{\bullet,m})=\K$. Thus, 
\[
^hE_{i,j}^1=\begin{cases}
 k^{\binom{3}{j}} \mbox{ if }i=0\\
0 \mbox{ otherwise.}
\end{cases}
\]

On the vertical spectral sequence, one has:
\[
^vE_{i,j}^1=\begin{cases}
 K_{i,m}\otimes_R H_0(K_{\bullet,f}) \mbox{ if }j=0\\
 K_{i,m}\otimes_R H_1(K_{\bullet,f})=\m. \mbox{ if }j=1\\
0 \mbox{ otherwise.}
\end{cases}
\]

\noindent Comparing the abutment of both spectral sequences, we have $^vE_{4,1}^2=H(\m.)_4=0$, $^vE_{3,1}^2=H(\m.)_3=0$ and  $^vE_{4,0}^2=H_4(K_{\bullet,m}\otimes_R H_0(K_{\bullet,f}))=0$.
\end{proof}

\begin{example0}\label{ex:(1,6)}
The homologies $H({\mathcal M})_{i}, i=1,2$ in~\eqref{eq:betaM} might be both nonzero.
For instance, let $\d=(1,6)$. Below we list all nonzero, non-Koszul
degree first Betti numbers for generic $\f$:
\begin{equation}\label{eq:beta2a}
 \beta_{1,(3,10)}= 1, \,  \beta_{1,(3,11)}= 4, \, \beta_{1,(4,10)}= 3, \, \beta_{1,(6,9)}= 2, \, \beta_{1,(1,18)}= 1.
 \end{equation}
The sum of all these numbers equals $11$. On the other side, the sum of the dimensions of all
$\dim_{\K} (H({\mathcal M})_{1,\a})$ equals $18$  and the sum of the dimensions of all $\dim_{\K} (H({\mathcal M})_{2,\a})$ 
equals $7$.  Indeed, $11=18-7$. We plot in Figure 1 below the bidegrees $\a$ with nonzero $\beta_{1,\a}$ in~\eqref{eq:beta2a}, together
with the curve $n_{(1,6)}(\a)=1$.
\vskip -.35in
\begin{figure}[ht]\label{FirstFig}
\begin{center}
\includegraphics[scale=0.28]{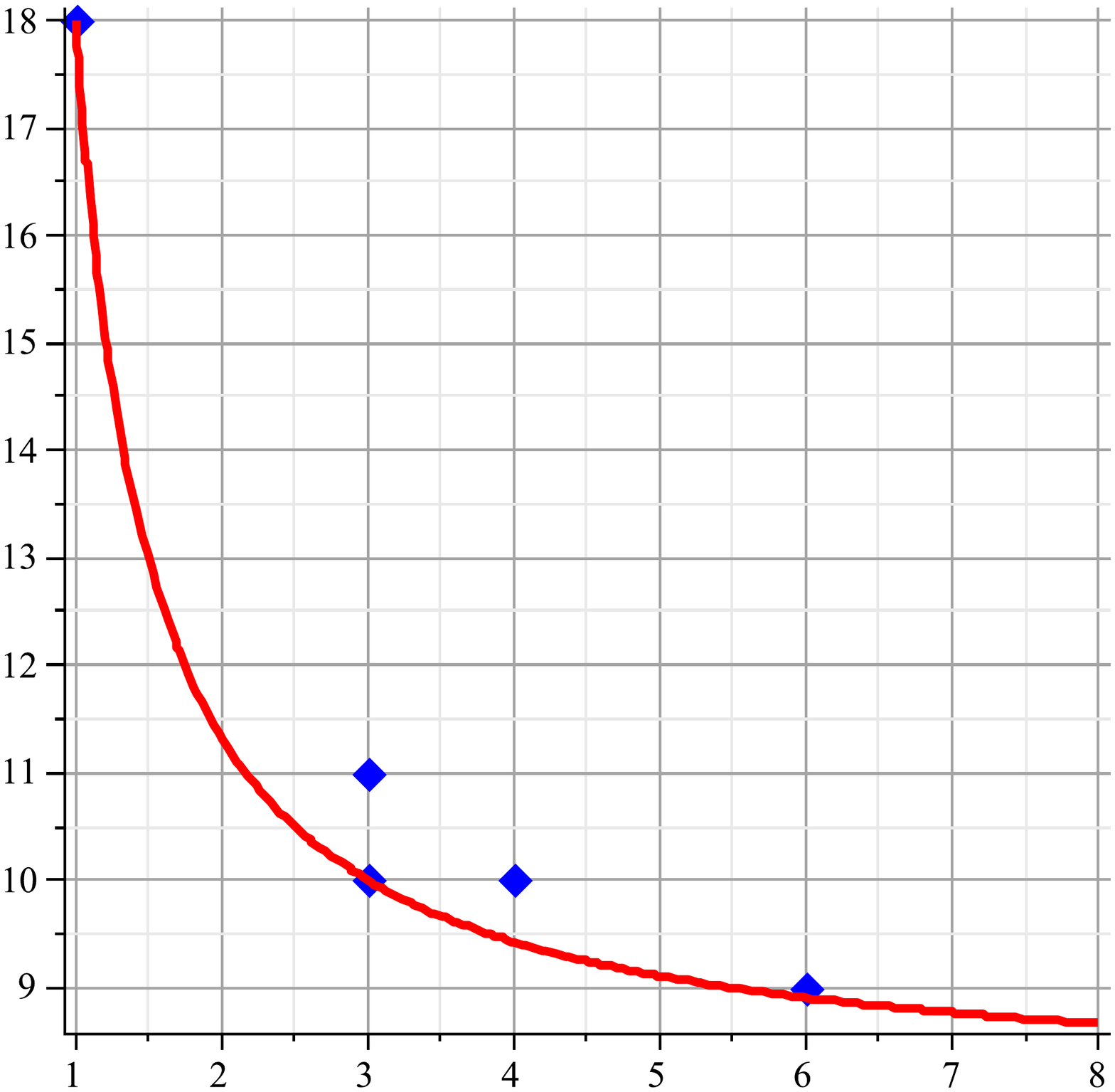}
\vskip -.5in
\caption{The generic case with $\d=(1,6)$}
\end{center}
\end{figure}
\noindent The values of $n_{(1,6)}(a)$ are given below, where the column on the left represents $a_1=0$, and the row on the bottom $a_2=0$.
\vskip -.35in

\begin{footnotesize}
\begin{center}
\begin{verbatim}
      | 0 3 0 0 0  0  0  0  0  0  0  |
      | 0 2 0 0 0  0  0  0  0  0  0  |
      | 0 1 0 0 0  0  0  0  0  0  0  |
      | 0 0 0 0 0  0  0  0  0  0  0  |
      | 0 0 0 1 2  3  4  5  6  7  8  |
      | 0 0 0 2 4  6  8  10 12 14 16 |
      | 0 0 0 3 6  9  12 15 18 21 24 |
      | 0 0 0 4 8  12 16 20 24 28 32 |
      | 0 0 0 5 10 15 20 25 30 35 40 |
      | 0 0 0 6 12 18 24 30 36 42 48 |
      | 0 0 0 1 5  9  13 17 21 25 29 |
      | 0 0 0 0 0  0  2  4  6  8  10 |
      | 0 0 0 0 0  0  0  0  0  0  0  |
      | 0 0 0 0 0  0  0  0  0  0  0  |
      | 0 0 0 0 0  0  0  0  0  0  0  |
      | 0 0 0 0 0  0  0  0  0  0  0  |
      | 0 0 0 0 0  0  0  0  0  0  0  |
      | 0 0 0 0 0  0  0  0  0  0  0  |
      | 0 0 0 0 0  0  0  0  0  0  0  |
      | 0 0 0 0 0  0  0  0  0  0  0  |
      | 0 0 0 0 0  0  0  0  0  0  0  |
\end{verbatim}
\end{center}
\end{footnotesize}

The values of the Betti numbers in~\eqref{eq:beta2a} can be deduced from Theorem~\ref{th:KoszulM} and Lemma~\ref{lem:generic}. A necessary condition is that $n_{(1,6)}(\a) \ge1$. For instance, we have that  
\[
n_{(1,6)}(3,10)=1, n_{(1,6)}(3,11)=6, \mbox{ and  }n_{(1,6)}(4,10)=5,
\]
so 
\[
\beta_{1,(3,11)}= n_{(1,6)}(3,11) - 2 n_{(1,6)}(3,10) = 6-2=4.
\]
Similarly, 
\[
\beta_{1,(4,10)}= n_{(1,6)}(4,10) - 2 n_{(1,6)}(3,10) = 5-2=3. 
\]
On the other side, 
\[
\beta_{1,(6,9)}= n_ {2,(6,9)}=2, \mbox{ and }\beta_{1,(1,18)}= n_{2,(1,18)}=1. 
\]
For a final example,  $\beta_{1,(5,10)}=0$ because 
\[
n_{2,(5,10)}=9 = 3 n_{2,(3,10)} + 2 (n_{2,(4,10)} - 2 n _{2,(3,10)}) = 2 n_{2,(4,10)} - n _{2,(3,10)}, 
\]
and $\beta_{1,(1,19)}=0$ because $n_{(1,6)}(1,19) < 2 n_{(1,6)}(1,18)$.
\end{example0}

\begin{example0}\label{ex:(1,42)}
Consider the bidegree $\d = (1,42)$ and $\f$ generic. We list all
bidegrees $\a$ with nonzero, non-Koszul first Betti number:
\begin{eqnarray*}\label{eq:beta2b} 
{}&(7, 67), (6, 68), (9, 66), (8, 67), (7, 68), (6, 69), (5, 70), (10,66), (4, 72), \\
		& (12, 65), (3, 75), (3, 76), (17, 64), (18, 64), (33, 63), (1, 126).
\end{eqnarray*}
This is the ordered list of the corresponding Betti numbers: $2, 3, 5, 8, 5, 8, 9, 3, 7, 6, 2, 3$, $3, 1, 2,1$.
In this example, there are minimal generators of the 
syzygy module in degrees $\a$, $\a-(1,0)$ and $\a-(0,1)$, for $\a=(7,68)$.
We focus on the bidegrees $(7, 68), (6, 68), (7,67)$, marked with solid diamonds in Figure 2. Note that we also show a few other bidegrees but we
do not display all bidegrees with nonzero, non-Koszul first Betti
number in the list above.
All these bidegrees must satisfy that  $n_{(1,42)}(\a) >0$. We also plot the curve $n_{(1,42)}(\a)=1$.
Again, the values of the Betti numbers can be deduced from Theorem~\ref{th:KoszulM} and Lemma~\ref{lem:generic}. 
\vskip -.35in
\begin{figure}[ht]
\begin{center}
\includegraphics[scale=0.28]{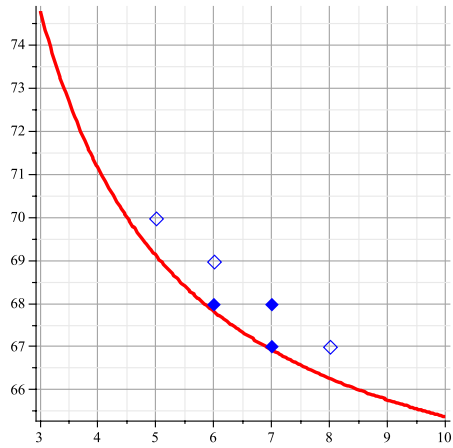}
\vskip -.5in
\caption{The generic case with $\d=(1,42)$}
\end{center}
\end{figure}
\end{example0}


\section{Factorization of sections of $\OPP(1,n)$ and the Segre variety $\Sigma_{1,n}$}
Recall the Segre variety $\Sigma_{r,s}$ is the image of the regular map
\[
\P^r \times \P^s \stackrel{\sigma_{r,s}}{\longrightarrow} \P^{rs+r+s}
\]
given by multiplication
\[
(x_0: \ldots : x_r),(y_0 : \ldots : y_s) \mapsto (x_0y_0: \ldots
:x_0y_s: x_1y_0:\ldots : x_ry_s).
\]

In general one has the following diagram

\begin{equation}\label{e1}
\xymatrix@C=8pt{
 \mathbb{P}(H^0(\Oc_\PP(1,i))) 
\!\times\! \mathbb{P}(H^0(\OPone(n\!-\!i))) \ar[r]\ar[dr]^{\psi_i} &
\mathbb{P}^{(2i+1)(n-i)+i+n+1} \ar[d]^{\pi}    \\
 & \mathbb{P}(H^0(\Oc_\PP(1,n)))}
\end{equation}

\noindent The composition of the Segre map
\[
\P(H^0(\OPP(1,1))) \times \P(H^0(\OPP(0,n-1))) = \P^3 \times \P^{n-1} \longrightarrow \P^{4n-1},
\]
with the projection $\pi$ onto $\P(H^0(\OPP(1,n)))$ is given with
respect to the basis 
$\{su^n, \cdots, sv^n, tu^n,\ldots, tv^n\}$ for $H^0(\OPP(1,n))$ by 
\[
\begin{array}{c}
(a_0: \cdots :a_3) \times (b_0: \cdots :b_{n-1}) \mapsto \\
(a_0b_0:a_0b_1+a_1b_0:a_0b_2+a_1b_1:\ldots a_0b_{n-1}+a_1b_{n-2}:a_1b_{n-1}:\\
a_2b_0:a_2b_1+a_3b_0:a_2b_2+a_3b_1:\ldots : a_2b_{n-1}+a_3b_{n-2}:a_3b_{n-1})
\end{array}
\]
For example, when $n=2$, the image of $\psi_1$ is a quartic hypersurface
\[
Q = \V({x}_{2}^{2} {x}_{3}^{2}-{x}_{1} {x}_{2} {x}_{3} {x}_{4}+{x}_{0} {x}_{2}
     {x}_{4}^{2}+{x}_{1}^{2} {x}_{3} {x}_{5}-2 {x}_{0} {x}_{2} {x}_{3} {x}_{5}-{x}_{0} {x}_{1} {x}_{4} {x}_{5}+{x}_{0}^{2} {x}_{5}^{2}).
\]

\begin{defn}
The image of the composite map $\psi_i$ is $\Sigma'_{2i+1,n-i}$.
\end{defn}

Therefore for $i=0$ we have $\Sigma_{1,n} = \Sigma'_{1,n}$, but for $i \ge 1$
the variety $\Sigma'_{2i+1,n-i}$ is a linear projection of
$\Sigma_{2i+1,n-i}$, with $\codim(\Sigma'_{2i+1,n-i})=n-i$. In particular $\Sigma'_{2i+1,n-i}$ is a hypersurface in $\Sigma'_{2i+3,n-i-1}$, and  
\[
\Sigma_{1,n}  = \Sigma'_{1,n} \subseteq \Sigma'_{3,n-1} \subseteq \Sigma'_{5,n-2}
\subseteq \cdots \subseteq \Sigma'_{2n-3,2} \subseteq \Sigma'_{2n-1,1} \subseteq \P^{2n+1}.
\]
If the basepoint free subspace $W \simeq \P^2$ is generic, then $W\cap \Sigma'_{2n-1,1}$ has dimension $1$,  $W\cap \Sigma'_{2n-3,2}$ is generically finite, and $W\cap \Sigma'_{2i+1,n-i}$ is empty for $i\neq n-1,n-2$.
\begin{lem}\label{BPfree}
If $W$ is a basepoint free three dimensional subspace of $\hh$, then 
$W$ is not contained in $\Sigma_{1, n}$. 
\end{lem}
\begin{proof}
If $W \subseteq \Sigma_{1, n}$, then by Theorem 9.22 of \cite{harris}, the only linear spaces contained in  
$\Sigma_{n,m}$ are those contained in one of the rulings. So 
$W \subseteq \Sigma_{1, n}$ would mean that $W = \{a\cdot l, b \cdot l, c \cdot l\}$ with
$\{a,b,c\}$ belonging to a fiber. But then $l$ is
either in $H^0(\OPP(1,0))$ or in $H^0(\OPone(n))$ and $W$ is not basepoint free.
\end{proof}
As $n$ grows, so do the possibilities for the 
intersection $W \cap \Sigma'_{2i+1,n-i}$.  However, for some 
special situations governed by geometry, there are resolutions 
which are independent of $n$, which we now explore. 
\subsection{Intersection with $\Sigma_{1,n}$} 
In the case where $n=2$, \cite{cds} shows that the only way in which $W$
can meet $\Sigma_{1,2}$ in a curve is if the curve is a smooth
conic. This phenomenon persists, but we need a bit more machinery.
\vskip .1in
\noindent {\bf Theorem} (Burau-Zeuge \cite{bz})
If $L$ is a linear space cutting each $n$-dimensional ruling of
$\Sigma_{1,n}$ in at most a single point, then $L \cap \Sigma_{1,n}$
is a rational normal curve.\newline
\begin{lem}~\label{conic1}
If $W$ meets $\Sigma_{1,n}$ in a curve, then it must be a smooth conic. 
\end{lem}
\begin{proof}
First, suppose $W$ contains a $\P^1$ fiber of $\Sigma_{1,n}$, so
that $W$ has basis $\{l_1s, l_2s, q\}$ with $l_i$ corresponding
to points on the $\P^1$. Then $\V(s,q)\ne \emptyset$ and $W$ is 
not basepoint free. Next, since $W$ is linear, it 
cannot meet a $\P^n$ fiber $F$ of
$\Sigma_{1,n}$ in more than two noncollinear points, for then
it would be contained in $F$ and hence violate
Lemma~\ref{BPfree}. If $W$ meets $F$ in two points, since
$W$ and $F$ are both linear, $W \cap F$ is a line $L$, and if $F$ is
the fiber over the point $l$ of $L$, then $W=\{al, bl, c\}$ and
since $\V(l,c)$ is nonempty on $\PP$, $W$ would have 
basepoints. In particular, $W$ can meet each $\P^n$ fiber in at
most a point, so by the result of Burau-Zeuge, $W \cap \Sigma_{1,n}$ is
a rational normal curve. As $W \simeq \P^2$,  the curve must
be a smooth conic.
\end{proof}
\begin{thm}\label{smoothconic}
$W \cap \Sigma_{1,n}$ is a smooth conic iff $I_W$ has a bidegree
$(3,n)$ first syzygy.
\end{thm}
\begin{proof}
Suppose $I_W$ has a minimal first syzygy of bidegree $(3,n)$, 
\[ 
a(s,t)\cdot f_0+b(s,t)\cdot f_1+c(s,t)\cdot f_2 = 0.
\] 
If $\langle a(s,t), b(s,t), c(s,t) \rangle \ne \langle s^2, st,
t^2\rangle$, then it must be generated by two bidegree $(2,0)$
quadrics $\{a(s,t),b(s,t)\}$ with no common factor. Changing basis for
$I_W$, the syzygy involves only $f_0, f_1$, which implies 
$f_0, f_1 = -b(s,t)g, a(s,t)g$ for some $g$. This is impossible 
by degree considerations, so after a change of basis for $W$, 
we may assume the $f_i$ satisfy
\[
s^2\cdot f_0+st\cdot f_1+t^2\cdot f_2 = 0 
\]
Now we switch perspective, and consider $[f_0,f_1,f_2]$ as a syzygy
on $[s^2,st,t^2]$. Since the syzygies on the latter space are
generated by the columns of 
\[
\left[\begin{matrix} t & 0\\
                    -s & -t\\ 
                     0 & s 
\end{matrix}\right],
\]
we have
\[
\begin{array}{ccc}
f_0 &=& t a_0\\
f_1 &=& sa_0 + ta_1\\
f_2 &=& sa_1,\\
\end{array}
\]
with the $a_i \in k[u,v]_n$. In particular, $\{a_0, a_1\}$ are a 
basepoint free pencil of $H^0(\OPone(n))$, and we may 
parameterize as in the proof of Theorem 4.1 of \cite{cds}
so that $W \cap \Sigma_{1,n}$ is a smooth conic. On the other hand, if $W \cap \Sigma_{1,n}$ is a smooth conic, the proof follows as in Theorem 4.1 of \cite{cds}.
\end{proof}
\begin{exm}
If $W= \{su^n, tv^n, sv^n+tu^n\}$, then since $\Sigma_{1,n}$ is given 
by the $2 \times 2$ minors of the matrix 
\[
\left[ 
\begin{array}{ccc}
x_0 & \cdots & x_n\\
x_{n+1} & \cdots &x_{2n+1}
\end{array} \right], 
\]
with $x_i =su^{n-i}v^i$ for $i\in \{0, \ldots, n\}$ and $tu^{n-i}v^i$
for $i\in \{n+1, \ldots, 2n+2\}$, so (dualizing) $W =\V(x_1, \ldots,
x_{n-1}, x_{n+2},\ldots x_{2n+1}, x_n-x_{n+1}) \subseteq \P(H^0(\OPP(1,n))^\vee)$,
and using coordinates $\{x_0,x_n,x_{2n+1}\}$ for $W$,  $W \cap \Sigma_{1,n} =W \cap 
\V(x_0x_{2n+1}-x_n^2)$. 
\end{exm}

\begin{remark}[Effective criterion]
Notice that given $I_W$, Theorem \ref{smoothconic} gives an effective
way to understand the geometry of $W$. Combined with the following
result we obtain a complete description of the minimal free resolution
of $I_W$ just by computing wether or not $I_W$ has a bidegree $(3,n)$ first syzygy.
\end{remark}

\subsection{Minimal free resolutions determined by the geometry of  $ W \cap \Sigma_{1,n}$ } 
We now examine situations where $ W \cap \Sigma_{1,n}$ has special
geometry.\newline

\begin{thm}\label{smoothconicres}
$W \cap \Sigma_{1,n}$ is a smooth conic iff $I_W$ has bigraded Betti numbers:
\begin{table}[ht]
\begin{supertabular}{|c|}
\hline 
  $0 \leftarrow I_W \leftarrow (-1,-n)^3 \stackrel{\partial_1}{\longleftarrow} \begin{array}{c}
( -1,-3n)\\
 \oplus \\
 (-2,-2n)^3\\
 \oplus \\
(-3,-n)\\
\end{array} \stackrel{\partial_2}{\longleftarrow} 
\begin{array}{c}
 (-2,-3n)^2\\
 \oplus \\
(-3,-2n)^2\\
 \end{array} \stackrel{\partial_3}{\longleftarrow} 
(-3,-3n)\leftarrow 0$ \\
\hline 
\end{supertabular}
\label{T2}
\end{table}
\end{thm}
\begin{proof}
In fact, we will show more, exhibiting the differentials in the
minimal free resolution. By Theorem~\ref{smoothconic}, we may choose the $f_i$ so that
\[
\begin{array}{ccc}
f_0 &=& ta_0\\
f_1 &=& sa_0+ta_1\\
f_2 &=& sa_1,\\
\end{array}
\]
Then the syzygy described by Lemma~\ref{Alicia} is $(a_1^2, -a_0a_1,a_0^2)$ and 
we also have the bidegree $(2,0)$ syzygy $(s^2,-st,t^2)$. Consider the
$\partial_i$ below. 
\[
\partial_1= \left[ \!
\begin{array}{ccccc}
a_1^2   &f_1   & f_2    & 0    & s^2\\
-a_0a_1 &-f_0  & 0      & f_2  &-st\\
a_0^2   & 0    &-f_0    & -f_1 &t^2 
\end{array}\! \right] 
\]
\[
\partial_2= \left[ \!
\begin{array}{cccc}
t    &s      & 0  & 0\\
-a_1 &0      & 0  &-s\\
a_0  & -a_1  &-s  & t\\
0    & a_0   &t  & 0\\
0    & 0    &a_1 & a_0 
 \end{array}\! \right] 
\]
\[
\partial_3= \left[ \!
\begin{array}{c}
s\\
-t\\
a_0\\
-a_1 
 \end{array}\! \right] 
\]
A check shows $\partial_i\partial_{i+1}=0$, exactness follows by Buchsbaum-Eisenbud \cite{be}.
\end{proof}

\noindent Example~\ref{11case} is a consequence of
Theorem~\ref{smoothconicres}, because for ${\bf d}=(1,1)$, $W$ is basepoint free iff it meets $\Sigma_{1,1}$ in
a smooth conic. When $W \cap \Sigma_{1,n}$ contains three distinct noncollinear points,
the resolution of $I_W$ is also completely determined.\newline

\pagebreak

\begin{thm}\label{3pointintersection}
If $|W \cap \Sigma_{1,n}|$ is finite and contains three noncollinear points,
then the bigraded Betti numbers of $I_W$ are
\begin{table}[ht]
\begin{center}
\begin{supertabular}{|c|}
\hline 
  $0 \leftarrow I_W \leftarrow (-1,-n)^3 \longleftarrow \begin{array}{c}
(-1,-3n)\\
 \oplus \\
 (-2,-2n)^3\\
 \oplus \\
(-3,-n-\mu)\\
\oplus \\
(-3,-2n+\mu)\\
\end{array} \longleftarrow 
\begin{array}{c}
 (-2,-3n)^2\\
 \oplus \\
(-3,-2n)^3\\
 \end{array} \longleftarrow 
(-3,-3n)\leftarrow 0$ \\
\hline 
\end{supertabular}
\end{center}
\label{T3}
\end{table}

\noindent with $0<\mu\leq \floor{n/2}$.
\end{thm}
\begin{proof}
As in the previous theorem, we will describe the differentials in the
minimal free resolution.
 Since $W \cap \Sigma_{1,n}$ contains three noncollinear points, we may
choose a basis so that $W = \{l_0g_0, l_1g_1, l_2g_2\}$ with the $l_i$ of bidegree
$(1,0)$ and the $g_i$ of bidegree $(0,n)$. If the $g_i$ are not
linearly independent, then changing basis we see that there are
constants $a,b,c,d$ with
\[W = \langle sg_0, tg_1, (as+bt)\cdot(cg_0+dg_1)\rangle = \langle
  sg_0,tg_1, acsg_0+bdtg_1\rangle,
\]
so $(bdt^2, acs^2, -st)$ is a bidegree $(2,0)$ syzygy on
$I_W$ and Theorem~\ref{smoothconic} applies. 

So we may assume $\{g_0,g_1,g_2\}$ are linearly independent; 
suppose the Hilbert-Burch matrix for $\{g_1, g_2, g_3\}$ 
has columns of degree $a$ and $n-\mu$.  In this case, in addition to the three Koszul
syzygies and the syzygy of Lemma~\ref{Alicia}, the Hilbert-Burch
syzygies can be lifted: if $(b_0, b_1, b_2)$ is a syzygy of degree
$a$ on the $g_i$,  then  $(l_1l_2b_0, l_0l_2b_1, l_0l_1b_2)$ is a syzygy
of bidegree $(3, n+\mu)$ on $I_W$, and similarly for the syzygy of degree
$n-\mu$. Note that $g_0=b_1c_2-c_1b_2, g_1=c_0b_2-b_0c_2,
g_2=b_0c_1-c_0b_1$. A priori, these need not be minimal, but by
constructing the remaining differentials and applying the
Buchsbaum-Eisenbud criterion, we will see that they are. Changing 
basis, we may assume $W$ has basis $=\{sg_0, tg_1, (as+bt)g_2\}$, 
hence the syzygy of Lemma~\ref{Alicia} takes the form
$(-ag_1g_2,-bg_0g_2,g_0g_1)$, and 
\[
\partial_1=\left[ \!
\begin{array}{cccccc}
-ag_1g_2 &tg_1  & (as+bt)g_2    & 0                &t(as+bt)b_0 &t(as+bt)c_0\\
-bg_0g_2 &-sg_0& 0                  & (as+bt)g_2 &s(as+bt)b_1 &s(as+bt)c_1\\
g_0g_1     & 0      &-sg_0            & -tg_1         &stb_2           &stc_2 
\end{array}\! \right] 
\]
A check shows that the two matrices below satisfy $\partial_2\partial_1=0$ and $\partial_3\partial_2=0$:
\[
\partial_2= \left[ \!
\begin{array}{ccccc}
t       &s        & 0           & 0      & 0\\
ag_2 &-bg_2 & as+bt    & 0      &0\\
0      & g_1    &  0          & t       &0\\
g_0  & 0        &  0           & 0      &s\\
0     & 0        &  c_2         &c_1    &c_0 \\
0     & 0         & -b_2       &-b_1   &-b_0 
 \end{array}\! \right] 
\partial_3= \left[ \!
\begin{array}{c}
s\\
-t\\
-g_2\\
g_1 \\
-g_0
 \end{array}\! \right] 
\]
Applying the Buchsbaum-Eisenbud criterion shows the complex is indeed exact. Since $n+\mu\leq 2n-\mu$, then it follows that $\mu\leq \floor{n/2}$.
\end{proof}

\begin{remark}If $n=2$ then $|W \cap \Sigma_{1,n}|$ is generically finite and generically contains three noncollinear points. For $n\geq 3$ this is not the case, so this closed condition is very restrictive. Moreover, for $n=2,3$, from Theorem \ref{3pointintersection}, one has that $\mu=1$.\end{remark}

\section{Higher Segre varieties} 
Since $\Sigma'_{2i+1,n-i}$ has codimension $n-i$, unless $i=n-1$ or $n-2$, the
intersection $W \cap \Sigma'_{2i+1,n-i}$ is generically
empty. Theorems~\ref{smoothconicres} and \ref{3pointintersection} illustrate the principle that when $W \cap
\Sigma'_{2i+1,n-i} \ne \emptyset$ and $i \le n-3$, special behavior can occur. The next theorem makes this explicit when 
$|W \cap \Sigma'_{2i+1,n-i}|$ is finite and contains at least three
noncollinear points.
\begin{thm}\label{liftSyz}
Suppose $W$ has basis $\{g_0h_0, g_1h_1, g_2h_2\}$ 
with
\[
  g_j \in H^0(\OPP(1,i)) \mbox{ and }h_j \in H^0(\OPone(n-i)), \mbox{
    with } 0 \le i \le n-1, 3 \le n.  \mbox{ Then }
\]
\begin{enumerate}
\item $\{h_0,h_1,h_2\}$ and $\{g_0,g_1,g_2\}$ are basepoint free.
\item A syzygy $\{a_0,a_1,a_2\}$ on the $h_i$ lifts to a syzygy
  (possibly non-minimal)
  \[
    \{g_1g_2a_0, g_0g_2a_1,g_0g_1a_2\}
  \]
  on
$I_W$, and similarly for a syzygy on the $g_i$. 
\item If $\{h_0,h_1,h_2\}$ is a pencil, then there is a bidegree
  $(3,2i+n)$ syzygy on $I_W$.
\item If $\{h_0,h_1,h_2\}$ is not a pencil, then it has a
  Hilbert-Burch matrix with columns of degrees $\{n-i-\mu, \mu\}$ in
  the $\{s,t\}$ variables. These give rise to syzygies of type $(2)$ above of bidegree
$(3,n+2i+\mu)$ and $(3,2n+i-\mu)$.
\end{enumerate}
\end{thm}
\begin{proof}
For $(1)$, if the $g_i$ or $h_i$ are not basepoint free, then 
neither is $W$, and for $(2)$, the result is immediate. For $(3)$, 
if the $h_i$ are a pencil, then $I_W = \langle g_0h_0, g_1h_1,
g_2(ah_0+bh_1)\rangle$ for constants $a,b$, 
and so $(ag_1g_2,bg_0g_2, -g_0g_1)$ is the desired syzygy,
and $(4)$ follows by applying $(2)$ to the Hilbert-Burch 
syzygies.
\end{proof}
The previous theorem deals with the situation where there is a basis for $W$
where all three elements factor in the same way. Even if only one or
two elements factor, the minimal free resolutions 
often behave differently from the generic case. Computations
suggest that all nongeneric behavior in the minimal free resolution 
stems from factorization:
\begin{conj}\label{nongenericR}
If the bigraded minimal free resolution of $I_W$ has 
nongeneric bigraded Betti numbers, then for some $i \le n-3$, 
$W \cap \Sigma'_{2i+1, n-i} \ne \emptyset$.
\end{conj}
\subsection{Intersection with $\Sigma'_{3, n-1}$}

We now examine the situation where the elements of $W$
factor into components of degree $(1,1)$ and $(0,n-1)$. The next
example shows that the converse to Conjecture~\ref{nongenericR} need
not hold.

\begin{exm}
Suppose $a,b,c$ are basepoint free elements of bidegree $(1,1)$, and $d,e,f$ 
are generic elements of bidegree $(0,4)$, with $I_W = \langle ad, be, cf 
\rangle$. A computation shows that the Betti numbers of $I_W$ are as
in the diagram below.
\begin{table}[ht]
\begin{center}
\begin{supertabular}{|c|}
\hline 
  $0 \leftarrow I_W \leftarrow (-1,-5)^3 \longleftarrow \begin{array}{c}
(-1,-15)\\
 \oplus \\
 (-2,-10)^3\\
 \oplus \\
(-3,-9)^5\\
\oplus \\
(-4,-8)^3\\
\oplus \\
(-9,-7)\\
\end{array} \longleftarrow 
\begin{array}{c}
 (-2,-15)^2\\
 \oplus \\
(-3,-10)^6\\
 \oplus \\
(-4,-9)^6\\
\oplus \\
(-9,-8)^2\\
 \end{array} \longleftarrow 
\begin{array}{c}
 (-3,-15)\\
 \oplus \\
(-4,-10)^3\\
\oplus \\
(-9,-9)\\
  \end{array}\leftarrow 0$ \\
\hline 
\end{supertabular}
\end{center}
\label{T5}
\end{table}

Lemma~\ref{Alicia} explains the $(1,15)$ syzygy, and there are three Koszul
syzygies. By construction, $W \cap \Sigma'_{3,4}$ consists of three points, 
and the bigraded Betti numbers for this example agree with the generic 
case; Theorem~\ref{1syz3m} explains the five $(3,9)$ syzygies, whereas 
Theorem~\ref{liftSyz} only accounts for two of them.
\end{exm}

\begin{thm}\label{smoothconicres2}
 Suppose $W = \Span \{g_0h_0, g_1h_1, g_2h_2\}$, with $g_i$ of
 degree $(1,1)$ and the $h_i$ a pencil of degree $(0,n-1)$, so that
 \[
   I_W =\langle g_0h_0, g_1h_1, g_2(ah_0+bh_1)\rangle.
 \]
If $W \cap \Sigma_{1,n}$ is empty and $W\cap \Sigma'_{3,n-1}$ contains
three noncollinear points, then there are minimal first syzygies of degrees
\[
 \{   ( -1,-3n), (-2,-2n)^3, (-3,\!-n\!-\!2), (-3,\!-2n\!+\!1)^2,
 (-6,\!-2n\!+\!2) \}
\]
\end{thm}
\begin{proof}
The syzygies of degree $(1,3n)$ and $(2,2n)$ are tautological.
The syzygy of degree $(3,n+2)$ can be explained by
Theorem~\ref{liftSyz} $(3)$, but it is more enlightening to treat the
syzygies of degree $(3,*)$ as a group. Write ${\bf f}$ as in
Theorem~\ref{1syz3m}:
\[
\{  sq_0+tq_1, sq_2+tq_3, sq_4+tq_5\}.
\]
Using that the $h_i$ are a pencil and expanding, we find that
\[
  \Span\{q_0, \ldots, q_5\} = \Span\{uh_0, vh_0,uh_1,vh_1\}.
\]
Since $\{h_0,h_1\}$ are basepoint free, they are a complete
intersection, so the Hilbert-Burch matrix for $\{uh_0, vh_0,uh_1,vh_1\}$ is
\[
  \left[ \!
\begin{array}{ccc}
  v & 0  &p_0   \\
-u & 0  &p_1 \\
  0     & v      & p_2 \\
0      &-u    &  p_3
\end{array}\! \right], 
\]
with the $p_i$ of degree $n-2$. In particular, the columns have degrees $\{b_1,b_2,b_3\} =
\{1,1,n-2\}$, which by Theorem~\ref{1syz3m} yields syzygies in
degrees
\[
  \{(3,2n-1), (3,2n-1), (3,n+2)\}
\]
By Corollary~\ref{123ID}, the seven syzygies constructed so far are
independent. We next prove there exists a unique first syzygy of
degree $(6,2n-2)$. Let $(s_0,s_1,s_2)$ be a syzygy on $I_W$, and
rewrite it as below
\begin{equation}\label{rewriteSyz}
  s_0g_0h_0+s_1g_1h_1+s_2g_2(ah_0+bh_1)=0 = (s_0g_0+as_2g_2)h_0+(s_1g_1+bs_2g_2)h_1.
\end{equation}
So $(s_0g_0+as_2g_2, s_1g_1+bs_2g_2)$ is a syzygy on the complete
intersection $(h_0,h_1)$, which implies that 
\[
  S = \left[ \!
\begin{array}{c}
 s_0g_0+as_2g_2\\
s_1g_1+bs_2g_2
 \end{array}\! \right], 
\]
is in the image of the Koszul syzygy $K$ on $\{h_0,h_1\}$:
\[
 K= \left[ \!
\begin{array}{c} h_1\\
-h_0
 \end{array}\! \right].
\]
One possibility is $S=0$, which leads to a syzygy of
Theorem~\ref{liftSyz} type 3, of bidegree $(3,n+2)$, which we have
accounted for, so we may suppose
$S$ is nonzero. This means $S = p\cdot K$ for some polynomial $p$.
Since the $h_i$ are degree $(0,n-1)$, the lowest possible degree for
the $s_i $ in the $(0,1)$ variables is $n-2$. We show that in degree
$(a,2n-2)$, there is a unique minimal syzygy of degree $(6,2n-2)$
which is not a multiple of the syzygy of degree $(3,n+2)$.

To see this, we write out Equation~\ref{rewriteSyz}, collecting the coefficients of
\[
 \{ u^{2n-2}, vu^{2n-1},\ldots, v^{2n-2}\}.
\]
Each $s_i$ is of degree $n-2$ in the $(0,1)$ variables, so there are
$3(n-1)$ columns, and we obtain a $2n-1 \times 3n-3$ matrix 
\[
 \mathcal{O}^{3n-3}_{\P^1}(-1)\stackrel{\psi}{\longrightarrow} \mathcal{O}^{2n-1}_{\P^1}.
 \]
The coefficients in the $(1,0)$ variables of the $s_i$ (written as
polynomials in the $(0,1)$ variables) correspond to elements of the kernel
of this matrix. The nonzero entries of $\psi$ come from the six linear forms in the $(1,0)$ variables,
obtained by writing the $(1,1)$ components $g_i$ as 
\[
g_i =a_iu+b_iv, \mbox{ with } \{a_i,b_i \} \in \K[s,t]_1.
\]
Since $W$ is basepoint free and $W \cap \Sigma_{1,n} = \emptyset$, the
cokernel  of $\psi$ is zero, so $\ker(\psi)$ is free of rank $n-2$,
with first Chern class $3-3n$. The key is that the unique syzygy of
degree $(3,n+2)$ generates $n-3$ independent syzygies of degree
$(3,2n-2)$, which follows from the computation 
\[
  h^0(\mathcal{O}_{\P^1} ((2n-2) -(n+2)))=h^0(\mathcal{O}_{\P^1}(n-4))=n-3.
\]
Since $\ker(\psi)$ has rank $n-2$, this means there is a single
additional element in the kernel, which has first Chern class 
\[
3(n-3) - (3n-3) = -6,
\]
yielding a unique first syzygy of degree $(6,2n-2)$.
\end{proof}
\begin{remark}If $W \cap \Sigma_{1,n} \ne \emptyset$, some of the $g_i$
  will factor. If all three factor, we are in the situation of
  Theorem~\ref{3pointintersection}. If only one or two factor, then
  the first syzygy of degree $(6,2n-2)$ changes to a syzygy of degree $(4,2n-2)$ if $|W \cap 
  \Sigma_{1,n}|=2$,  and to a syzygy of degree $(5,2n-2)$ if $|W \cap   \Sigma_{1,n}|=1$.
 \end{remark} 
\begin{exm}\label{WholeReshigherSegre}
With the hypotheses of Theorem~\ref{smoothconicres2}, computations
suggest that the bigraded Betti numbers are 
\begin{table}[ht]
\begin{center}
\begin{supertabular}{|c|}
\hline 
  $0 \leftarrow I_W \leftarrow (-1,\!-n)^3 \stackrel{\partial_1}{\longleftarrow\!\!\!} \begin{array}{c}
( -1,-3n)\\
 \oplus \\
 (-2,-2n)^3\\
 \oplus \\
(-3,\!-n\!-\!2)\\
 \oplus \\
(-3,\!-2n\!+\!1)^2\\
 \oplus \\
(-6,\!-2n\!+\!2)\\
\end{array} \stackrel{\partial_2}{\!\!\longleftarrow\!\!} 
\begin{array}{c}
 (-2,-3n)^2\\
 \oplus \\
(-3,-2n)^4\\
\oplus \\
(-6,\!-2n\!+\!1)^2\\
 \end{array} \stackrel{\partial_3}{\!\!\longleftarrow\!\!} 
\begin{array}{c}
 (-3,\!-3n)\\
 \oplus \\
(-6,\!-2n)\\
 \end{array}
\leftarrow 0$ \\
\hline 
\end{supertabular}
\end{center}
\label{T6}
\end{table}

\noindent Note that the numbers
\[
\beta_{2,\{*,-3n\}} \mbox{ and } \beta_{3,\{*,-3n\}} 
\]
are explained by Proposition~\ref{prop3.2}; one way to prove the diagram above is the correct betti table would be to determine explicitly the
$(6,2n-2)$ syzygy, and then write down the differentials and apply
the Buchsbaum-Eisenbud criterion for exactness.
\end{exm}

\subsection{Connections to the Hurwitz discriminant and Sylvester map.}

We consider the case $d=(1,n)$ with $n=5$ at the particular bidegree $(3,8)$. The aim is to point out the tip of the iceberg of the relation between non generic
rank behavior in the matrix of $\phi_1$ from Section~\ref{sec:generalcase}
and the different intersections
of $W$ with the variety of elementary tensors associated to all the decompositions of $(1,n) = (1,i)+(0,n-i)$.

We set $n=5$. According to Theorem~\ref{1syz3m},
when the $f_i$ are generic there are syzygies in degree $(3,m)$ for $m \ge 2\cdot 5 -0-1=9$ ($\kappa =0$ for $n=5$).
But in this bidegree, we can read in the table of Section~\ref{sec:generalcase}
that the matrix of $\phi_1$ has always full rank in the basepoint free case.

We then move one step to bidegree $(3,8) = (3, 2 \cdot 5-2)$
and we try to detect the existence of nontrivial syzygies there.
In this bidegree $\phi_1:  R_{(0,5)} \to R^3_{(1,0)}$, so we get a $6 \times 6$ matrix $M$, constructed as follows,
according to Example~\ref{ex:maps}.
Write
\[
  f_\ell = s \left(\sum_{j=0}^5 p_{\ell j} u^j v^{5-j}\right) + t
  \left(\sum_{j=0}^5 q_{\ell j} u^j v^{5-j}\right), \mbox{ for  }\ell=0,1,2.
  \]
Then
\[ M \, = \, 
\left(\begin{array}{ccc}
p_{00} & \dots  & p_{05} \\
q_{00} & \dots  & q_{05}\\ \hline
p_{10} & \dots  & p_{15} \\
q_{10} & \dots  & q_{15}\\ \hline
p_{20} & \dots  & p_{25} \\
q_{20} & \dots  & q_{25}
\end{array}\right).
\]
Let $S_{(3,8)}$ denote the hypersurface ${\bf V}(\det(M))$. 
Consider the $9$-dimensional variety $\Sigma'_{7,2} \subset \P^{11}$
of bihomogeneous polynomials factoring as the product of a degree $(1,3)$ polynomial and a degree $(0,2)$ polynomial. We expect $\dim(\Sigma'_{7,2} \cap W)=0$.  
It would be interesting to relate the geometry of $S_{(3,8)}$ to special features of this intersection (for example, finitely 
many points but with {\em special} properties, or a curve). Note that $W$ lives in the Grassmanian of planes in $\P^{11}$ and 
there is a polynomial, called the {\em Hurwitz discriminant} by Sturmfels~\cite{Sturm}, which vanishes whenever the 
intersection  $\Sigma'_{7,2} \cap W$ does not consist of ${\rm deg}(\Sigma'_{7,2})$ many points. 

\subsection{Concluding remarks} We close with a number of questions:
\begin{enumerate}
\item What happens when there are many ``low degree'' first syzygies?
  As shown in \cite{ds} and \cite{ssv}, linear first syzygies impose
  strong constraints.
\item $W$ is a point of $\mathbb{G}(2,2n+1)$. How does the Schubert
  cell structure impact the free resolution of $I_W$?
  \item What happens in other bidegrees? For other toric surfaces?
\item Are there special cases such as in Theorem~\ref{smoothconicres},
Theorem~\ref{3pointintersection}, Theorem~\ref{smoothconicres2} which 
are of interest to the geometric modeling community?
\item For computation of toric cohomology, it is sufficient to have a
  complex with homology supported in $B$, rather than an exact
  sequence. This is studied by Berkesch-Erman-Smith in \cite{bes}, and is a very active area
  of research.
\end{enumerate}

\noindent{\bf Acknowledgments.} Most of this paper was written while
the third author was visiting Universidad de Buenos Aires on a Fulbright
grant, and he thanks the Fulbright foundation for support and his 
hosts for providing a wonderful visit. All computations were done
using {\tt Macaulay2} \cite{M2}.

\bibliographystyle{amsplain}

\begin{thebibliography}{10}

\bibitem{acd} {\scshape A.\ Aramova, K.\ Crona, and E.\ De Negri}, 
Bigeneric initial ideals, diagonal subalgebras and bigraded Hilbert functions, 
\textit{J.\ Pure Appl.\ Algebra }{\bf 150} (2000), 215--235.

\bibitem{bes} {\scshape C.\ Berkesch, D.\ Erman, G.G.\ Smith},
 Virtual resolutions for a product of projective spaces,
\textit{Algebraic Geometry}, {\bf 7} (2020), 460-481.

\bibitem{bdd} {\scshape N.\ Botbol, A.\ Dickenstein, M.\ Dohm},
 Matrix representations for toric parametrizations,
\textit{Comput. Aided Geom. D.} {\bf 26}  (2009),  757--771. 

\bibitem{be} {\scshape D.\ Buchsbaum, D.\ Eisenbud}, 
What makes a complex exact?,
\textit{J. Algebra} {\bf 25} (1973), 259--268.

\bibitem{bz} {\scshape W.\ Burau, J.\ Zeuge},
\"Uber den Zusammenhang zwischen den Partitionen einer 
nat\"urlichen Zahl und den linearen Schnitten der einfachsten 
Segremannigfaltigkeiten. 
\textit{J. Reine Angew. Math} {\bf 274-75} (1975), 104-111. 

\bibitem{bc} {\scshape L.\ Bus\'e, M.\ Chardin},
Implicitizing rational hypersurfaces using approximation complexes,
\textit{J. Symb. Comput.} {\bf 40} (2005), 1150--1168. 

\bibitem{c} {\scshape D.\ Cox}, 
Curves, surfaces and  syzygies, in 
``Topics in algebraic geometry and geometric modeling",
\textit{Contemp. Math.} {\bf 334} (2003)  131--150.

\bibitem{cds} {\scshape D.\ Cox, A.\ Dickenstein, H.\ Schenck}, 
A case study in bigraded commutative algebra, in 
``Syzygies and Hilbert Functions", edited by Irena Peeva, 
Lecture notes in Pure and Applied Mathematics 254, (2007), 67--112.

\bibitem{cgz} {\scshape D.\ Cox, R. \ Goldman, M. \ Zhang}, On the
 validity of implicitization by moving quadrics for rational surfaces
with no basepoints, \textit{J. Symb. Comput.} {\bf 29} (2000), 419--440.

\bibitem{d} {\scshape W.L.F.\ Degen}, The types of rational $(2,1)$-B\'ezier surfaces. {\em Comput. Aided Geom. D.} {\bf 16}  (1999), 639--648.


\bibitem{ds}  {\scshape E. Duarte, H. Schenck},
  Tensor product surfaces and linear syzygies,
  \textit{P. Am. Math. Soc.}, {\bf 144} (2016), 65--72.

\bibitem{ebig} {\scshape D.\ Eisenbud}, \textit{Commutative Algebra with 
a view towards Algebraic Geometry}, Springer-Verlag,
Berlin-Heidelberg-New York, 1995.

\bibitem{GeomSyz} {\scshape D.\ Eisenbud}, \textit{Geometry of
    Syzygies}, Springer-Verlag, Berlin-Heidelberg-New York, 2005.

\bibitem{egl} {\scshape M.\ Elkadi, A.\ Galligo and  T.\ H.\ L\^{e}},
Parametrized surfaces in $\P^3$ of bidegree $(1,2)$,
\textit{Proceedings of the 2004 International Symposium on Symbolic and
Algebraic Computation}, ACM, New York, 2004, 141--148.

\bibitem{FL} {\scshape R. Fr\"oberg, S. Lundqvist}, 
Questions and conjectures on extremal Hilbert series. 
\textit{Rev. Union Mat. Argent.} {\bf 59} (2018), 415--429. 

\bibitem{gl} {\scshape A.\ Galligo, T.\ H.\ L\^{e}},
General classification of $(1,2)$ parametric surfaces in ${\mathbb P}^3$, 
in ``Geometric modeling and algebraic geometry'', Springer, Berlin, 
(2008) 93--113.


\bibitem{M2} {\scshape D.\ Grayson, M.\ Stillman},
 \textit{Macaulay2, a software system for research in algebraic
   geometry},
 Available at http://www.math.uiuc.edu/Macaulay2/.
  
\bibitem{harris} {\scshape J.\ Harris}, \textit{Algebraic Geometry, A
First Course}, Springer-Verlag, Berlin-Heidelberg-New York, 1992.

\bibitem{h} {\scshape R.\ Hartshorne}, \textit{Algebraic Geometry},
Springer-Verlag, Berlin-Heidelberg-New York, 1977.

\bibitem{hsv1} {\scshape J.\ Herzog, A.\ Simis, W.\ Vasconcelos}
 Approximation complexes of blowing-up rings,
\textit{J. Algebra} {\bf 74} (1982), 466--493.

\bibitem{hsv2} {\scshape J.\ Herzog, A.\ Simis, W.\ Vasconcelos}
Approximation complexes of blowing-up rings II, 
\textit{J. Algebra} {\bf 82} (1983), 53--83.

\bibitem{hw} {\scshape J.~W.\ Hoffman and H.~H.\ Wang}, 
Castelnuovo-Mumford regularity in biprojective spaces, 
{\em Adv. Geom.} {\bf 4} (2004), 513--536.

\bibitem{MS} {\scshape D.\ Maclagan, G.G. Smith}, 
Multigraded Castelnuovo-Mumford regularity,
\textit{J. Reine Angew. Math.}, {\bf 57} (2004), 179-212.

\bibitem{r} {\scshape T.\ R\"omer}, 
Homological properties of bigraded algebras, 
{\em Illinois J. Math.}, {\bf 45} (2001), 1361--1376.
 
\bibitem{ssv} {\scshape H.\ Schenck, A.\ Seceleanu, J.\ Validashti},
Syzygies and singularities of tensor product surfaces of bidegree $(2,1)$, 
\textit{Math. Comp.} {\bf 83} (2014), 1337--1372.

\bibitem{sc} {\scshape T.\ W.\ Sederberg, F.\ Chen}, Implicitization
using moving curves and surfaces, in {\sl Proceedings of
SIGGRAPH}, 1995, 301--308.

\bibitem{Sturm} {\scshape B. Sturmfels}, 
The Hurwitz form of a projective variety,
\textit{J. Symb. Comput.}, {\bf 79} (2017), 186--196.

\bibitem{z1} {\scshape S.\ Zube},
Correspondence and $(2,1)$-B\'ezier surfaces,
\textit{Lith. Math. J.} {\bf 43}  (2003), 83--102.

\bibitem{z2} {\scshape S.\ Zube},
Bidegree $(2,1)$ parametrizable surfaces in ${\mathbb P}^3$,
\textit{Lith. Math. J.} {\bf 38}  (1998), 291--308.

\end{thebibliography}

\end{document}